\documentclass[11pt,leqno]{amsart}

\usepackage{latexsym}
\usepackage{amssymb}
\usepackage{amsmath}
\usepackage{amsthm}
\usepackage{verbatim}
\usepackage{pdfsync}

\usepackage[margin=1in]{geometry}

\numberwithin{equation}{section}  

\begin{document}


\newcommand{\zxz}[4]{\begin{pmatrix} #1 & #2 \\ #3 & #4 \end{pmatrix}}
\newcommand{\abcd}{\zxz{a}{b}{c}{d}}
\newcommand{\kzxz}[4]{\left(\begin{smallmatrix} #1 & #2 \\ #3 &
#4\end{smallmatrix}\right) }
\newcommand{\kabcd}{\kzxz{a}{b}{c}{d}}




\newcommand{\A}{{\mathbb A}}
\newcommand{\C}{{\mathbb C}}
\providecommand{\E}{{\mathbb E}}
\newcommand{\F}{{\mathbb F}}
\newcommand{\G}{{\mathbb G}}
\newcommand{\R}{{\mathbb R}}
\newcommand{\Q}{{\mathbb Q}}
\newcommand{\X}{{\mathbb X}}
\newcommand{\Z}{{\mathbb Z}}
\newcommand{\HZ}{\widehat{\Z}}


\newcommand{\rom}[1]{\text{\rm #1}}
\renewcommand{\roman}{\rm}

\newcommand{\Aut}{\text{\rm Aut}}
\newcommand{\CH}{\widehat{\text{\rm CH}}}
\newcommand{\cha}{{\text{\rm char}}}
\newcommand{\CHe}{\text{\rm CHeeg}}
\newcommand{\degh}{\widehat{\text{\rm deg}}}
\newcommand{\degH}{\widehat{\text{\rm deg}}}    
\newcommand{\diag}{{\text{\rm diag}}}
\newcommand{\Diff}{\text{\rm Diff}}
\newcommand{\disc}{\text{\rm discr}}
\renewcommand{\div}{\text{\rm div}}
\newcommand{\divh}{\widehat{\text{\rm div}}}
\newcommand{\DS}{\text{\rm DS}}
\newcommand{\Ei}{\text{\rm Ei}}
\newcommand{\End}{\text{\rm End}}
\newcommand{\ev}{{\text{\rm ev}}}
\newcommand{\Gal}{\text{\rm Gal}}
\newcommand{\GL}{\text{\rm GL}}
\newcommand{\GSpin}{\text{\rm GSpin}}
\newcommand{\Hom}{\text{\rm Hom}}
\newcommand{\hor}{{\text{\rm horiz}}}
\newcommand{\id}{\text{\rm id}}
\newcommand{\im}{\text{\rm im}}
\renewcommand{\Im}{\text{\rm Im}}
\newcommand{\inv}{{\text{\rm inv}}}
\newcommand{\Jac}{\text{\rm Jac}}
\newcommand{\Leray}{{\mathrm L}}
\newcommand{\Lie}{\text{\rm Lie}}
\newcommand{\Mp}{\text{\rm Mp}}
\newcommand{\mult}{\text{\rm mult}}
\newcommand{\MW}{\text{\rm MW}}
\newcommand{\MWt}{\widetilde{\MW}}
\newcommand{\new}{\text{\rm new}}
\newcommand{\Nm}{\text{\rm Nm}}
\newcommand{\ord}{\text{\rm ord}}
\newcommand{\PGL}{\text{\rm PGL}}
\newcommand{\Pic}{\text{\rm Pic}}
\newcommand{\Pich}{\widehat{\text{\rm Pic}}}
\newcommand{\pr}{\text{\rm pr}}
\newcommand{\ra}{\text{\rm ra}}
\newcommand{\Rao}{\mathrm R}
\renewcommand{\Re}{\text{\rm Re}}
\newcommand{\sgn}{\text{\rm sgn}}
\newcommand{\sig}{\text{\rm sig}}
\newcommand{\SL}{\text{\rm SL}}
\newcommand{\SO}{\text{\rm SO}}
\newcommand{\Sp}{\text{\rm Sp}}
\newcommand{\Spec}{\text{\rm Spec}\, }
\newcommand{\Spf}{\text{\rm Spf}}
\newcommand{\supp}{\text{\rm supp}}
\newcommand{\Sym}{{\text{\rm Sym}}}
\newcommand{\tr}{\text{\rm tr}}
\newcommand{\type}{\text{\rm type}}
\newcommand{\Ver}{\text{\rm Vert}}
\newcommand{\vol}{\text{\rm vol}}
\newcommand{\Wald}{\text{\rm Wald}}


\newcommand{\Cal}{\mathcal}     

\newcommand{\AHH}{\hat{\Cal A}}   
\newcommand{\CHH}{\hat{\Cal C}}
\newcommand{\MM}{\Cal D}          
\newcommand{\MMb}{\MM^\bullet}
\newcommand{\ssplit}{\text{\bf split}}
\newcommand{\whcc}{\widehat{\Cal C}}
\newcommand{\CO}{\mathcal O}
\newcommand{\COH}{\widehat{\CO}}
\newcommand{\M}{\Cal M}
\newcommand{\OB}{\Cal O_B}
\newcommand{\XX}{\mathcal X}
\newcommand{\bXX}{\bar\XX}
\newcommand{\wc}{\hat{\Cal C}}
\newcommand{\wch}{\wc^{\text{\rm hor}}}
\newcommand{\ZZ}{\Cal Z}
\newcommand{\ZH}{\widehat{\Cal Z}}   
\newcommand{\Zh}{\widehat{\Cal Z}}
\newcommand{\ZZh}{\ZZ^{\text{\rm hor}}}
\newcommand{\ZZv}{\ZZ^{\text{\rm ver}}}
\newcommand{\ZZhh}{\Zh^{\text{\rm hor}}}
\newcommand{\ZZhv}{\Zh^{\text{\rm ver}}}


\newcommand{\nass}{\noalign{\smallskip}}
\newcommand{\snass}{\noalign{\vskip 2pt}}
\newcommand{\tent}[1]{ \vphantom{\vbox to #1pt{}} }   


\newcommand{\scr}{\scriptstyle}
\newcommand{\disp}{\displaystyle}

\font\cute=cmitt10 at 12pt
\font\smallcute=cmitt10 at 9pt
\newcommand{\kay}{{\text{\cute k}}}
\newcommand{\smallkay}{{\text{\smallcute k}}}

\renewcommand{\a}{\alpha}
\renewcommand{\b}{\beta}
\newcommand{\e}{\epsilon}
\renewcommand{\l}{\lambda}
\renewcommand{\L}{\Lambda}
\renewcommand{\o}{\omega}
\renewcommand{\O}{\Omega}
\renewcommand{\P}{\Phi}
\newcommand{\ph}{\varphi}
\newcommand{\phih}{\widehat{\phi}}
\newcommand{\wphi}{\widehat{\phi}}
\newcommand{\phit}{\widetilde{\phi}}
\newcommand{\s}{\sigma}
\newcommand{\vth}{\vartheta}


%

\newcommand{\Pt}{P}
\newcommand{\Ph}{\P}
\newcommand{\Pht}{\tilde \P}   
\newcommand{\Kt}{K}           
\newcommand{\Mt}{M}

\newcommand{\pht}{\widetilde{\phi}}
\newcommand{\It}{I}
\newcommand{\Jt}{\widetilde{J}}
\newcommand{\lt}{\widetilde{\l}}
\newcommand{\vp}{\varpi}

\newcommand{\bom}{{\boldsymbol{\o}}}
\newcommand{\hbom}{\widehat{\bom}}
\newcommand{\ff}{{\bold f}}
\newcommand{\fsp}{\boldsymbol{f}_{\!\rm sp}}
\newcommand{\fev}{\boldsymbol{f}_{\!\rm ev}}
\newcommand{\fb}{\boldsymbol{f}}
\newcommand{\J}{\und{J}'}
\newcommand{\JJ}{\bold J'}
\newcommand{\V}{\bold V}
\newcommand{\xx}{\bold x}

\newcommand{\g}{{\mathfrak g}}
\renewcommand{\H}{\mathfrak H}


\newcommand{\back}{\backslash}
\newcommand{\CT}[1]{\operatornamewithlimits{CT}_{#1}}
\renewcommand{\d}{\partial}
\newcommand{\db}{\bar\partial}
\newcommand{\dbar}{\bar{\partial}}
\newcommand{\gs}[2]{\langle \,#1,#2\,\rangle}
\newcommand{\Gt}{G}
\newcommand{\hfal}{h_{\text{\rm Fal}}}
\newcommand{\II}{\int^\bullet}
\newcommand{\isoarrow}{\ {\overset{\sim}{\longrightarrow}}\ }
\newcommand{\lisoarrow}{\ {\overset{\sim}{\longleftarrow}}\ }
\newcommand{\limdir}[1]{\underset{\underset{#1}{\rightarrow}}{\lim}}
\newcommand{\lan}{\operatorname{\langle}\hskip .5pt}
\newcommand{\ran}{\,\operatorname{\rangle}}
\newcommand{\lra}{\longrightarrow}
\newcommand{\doublelra}{\ {\overset{\scr\lra}{\scr\lra}}\ }
\newcommand{\nat}{\natural}
\newcommand{\notmid}{\mkern-5mu\not\mkern5mu\mid}
\newcommand{\Optoc}{\text{\rm Opt}(O_{c^2d},O_B)}
\newcommand{\psim}{\psi^{-}}
\newcommand{\qeq}{\ \overset{??}{=}\ }
\newcommand{\sh}{\sharp}
\newcommand{\thCH}{\theta^{\text{\rm ar}}}
\newcommand{\wht}{\widehat{\theta}}     
\newcommand{\triv}{1\!\!1}
\renewcommand{\tt}{\otimes}
\newcommand{\und}[1]{\underline{#1}}
\newcommand{\z}{z}  

\newcommand{\thMW}{\theta^{\text{\rm ar}}}
\newcommand{\tph}{\widetilde{\widehat\phi_1}}
\newcommand{\Pet}{\text{\rm Pet}}





\newcommand{\thing}{ \raisebox{-6.4pt}{$\tilde{\tilde{}}$}  }   
\newcommand{\smallthing}{ \raisebox{-4.4pt}{$\scr\tilde{\tilde{}}$}  }
\newcommand{\ttilde}[1]{\overset{\smash{\thing}}{#1}}
\newcommand{\smallttilde}[1]{\overset{\smash{\smallthing}}{#1}}
\newcommand{\downhookarrow}{\hbox{$\downarrow\hskip -6.1pt\raisebox{6pt}{$\cap$}$}}


\providecommand{\bysame}{\makebox[3em]{\hrulefill}\thinspace}   
\newcommand{\hfb}{\hfill\break}
\newcommand{\margincom}[1]{\marginpar{\bf\raggedright #1}}
\newcommand{\Sec}{\S}


\numberwithin{equation}{section}
\setcounter{section}{0}
\setcounter{MaxMatrixCols}{15}


\newtheorem{theo}{Theorem}[section]
\newtheorem{lem}[theo]{Lemma}
\newtheorem{prop}[theo]{Proposition}
\newtheorem{cor}[theo]{Corollary}
\newtheorem*{main}{Main Theorem}
\newtheorem*{maina}{Main Theorem A}
\newtheorem*{mainb}{Main Theorem B}
\newtheorem*{atheo}{Theorem A}
\newtheorem*{btheo}{Theorem B}
\newtheorem*{ctheo}{Theorem C}
\newtheorem*{dtheo}{Theorem D}
\newtheorem*{etheo}{Theorem E}
\newtheorem{conj}[theo]{Conjecture}
\theoremstyle{remark}
\newtheorem{rem}[theo]{Remark}      
\newtheorem{example}[theo]{Example}
\theoremstyle{definition}
\newtheorem{defn}[theo]{Definition}

\renewcommand{\E}{{\mathbb E}}

\newcommand{\OO}{\text{\rm O}}
\newcommand{\UU}{\text{\rm U}}

\newcommand{\OK}{O_{\smallkay}}
\newcommand{\DI}{\mathcal D^{-1}}

\newcommand{\pre}{\text{\rm pre}}

\newcommand{\Bor}{\text{\rm Bor}}
\newcommand{\Rel}{\text{\rm Rel}}
\newcommand{\rel}{\text{\rm rel}}
\newcommand{\Res}{\text{\rm Res}}
\newcommand{\TG}{\widetilde{G}}

\newcommand{\OL}{O_{\Lambda}}
\newcommand{\OLB}{O_{\Lambda,B}}

\newcommand{\p}{\varpi}

\newcommand{\cutter}{\vskip .1in\hrule\vskip .1in}

\parindent=0pt
\parskip=10pt
\baselineskip=13pt

\newcommand{\PP}{\mathcal P}
\renewcommand{\OO}{\mathcal O}
\newcommand{\BB}{\mathbb B}
\newcommand{\OBB}{O_{\BB}}
\newcommand{\Max}{\text{\rm Max}}
\newcommand{\Opt}{\text{\rm Opt}}
\newcommand{\OH}{O_H}

\newcommand{\phhat}{\widehat{\phi}}
\newcommand{\thetahat}{\widehat{\theta}}

\newcommand{\lbold}{\text{\boldmath$\l$\unboldmath}}
\newcommand{\abold}{\text{\boldmath$a$\unboldmath}}
\newcommand{\cbold}{\text{\boldmath$c$\unboldmath}}
\newcommand{\aabold}{\text{\boldmath$\a$\unboldmath}}
\newcommand{\gbold}{\text{\boldmath$g$\unboldmath}}
\newcommand{\obold}{\text{\boldmath$\o$\unboldmath}}
\newcommand{\fbold}{\text{\boldmath$f$\unboldmath}}
\newcommand{\rbold}{\text{\boldmath$r$\unboldmath}}
\newcommand{\ffbold}{\und{\fbold}}

\newcommand{\deltaBB}{\delta_{\BB}}
\newcommand{\kappaBB}{\kappa_{\BB}}
\newcommand{\aboldBB}{\abold_{\BB}}
\newcommand{\lboldBB}{\lbold_{\BB}}
\newcommand{\gboldBB}{\gbold_{\BB}}
\newcommand{\bbold}{\text{\boldmath$\b$\unboldmath}}

\newcommand{\fff}{\phi}

\newcommand{\spp}{\text{\rm sp}}


\newcommand{\bb}{\frak b}

\newcommand{\bbbold}{\text{\boldmath$b$\unboldmath}}

\renewcommand{\ll}{\,\frak l}
\newcommand{\uC}{\underline{\Cal C}}
\newcommand{\uZZ}{\underline{\ZZ}}
\newcommand{\B}{\mathbb B}
\newcommand{\CL}{\text{\rm Cl}}

\newcommand{\pp}{\frak p}

\newcommand{\OKp}{O_{\smallkay,p}}

\renewcommand{\top}{\text{\rm top}}

\newcommand{\bF}{\bar{\mathbb F}_p}

\newcommand{\beq}{\begin{equation}}
\newcommand{\eeq}{\end{equation}}


\newcommand{\Dl}{\Delta(\l)}
\newcommand{\mm}{{\bold m}}

\title{On the pullback of an arithmetic theta function}

\author{Stephen Kudla
\medskip\\
and
\medskip\\
Tonghai Yang}

\thanks{
The second author is partially supported by grants NSF DMS-0855901
and NSFC-10628103.\\
The first author was supported by an NSERC Discovery grant.}

\address{
Department of Mathematics, University of Toronto\hfill
\medskip
\break
Department of Mathematics, University of Wisconsin
}

\maketitle

\section{\bf Introduction}

In this paper, we consider the relation between the
simplest types of arithmetic theta series, those associated to the cycles on the moduli space $\Cal C$
of elliptic curves with CM by the ring of integers $\OK$ in an imaginary quadratic field $\kay$,
on the one hand,
and those associated to cycles on the arithmetic surface $\M$ parametrizing $2$-dimensional
abelian varieties with an action of the maximal order $O_B$ in an indefinite quaternion algebra
$B$ over $\Q$, on the other.

To be more precise, let $\Cal C$ be  the moduli stack of elliptic curves
$(E,\iota)$ with $\OK$ action, so that $\Cal C$
is an arithmetic curve over $\Spec(\OK)$, \cite{KRYtiny}.
Let $L(E,\iota)$ be the space of special endomorphisms,  i.e., endomorphisms $j$ of $E$ such that
$j\circ \iota(\a) = \iota(\a^\s)\circ j$ for all $\a\in \OK$, where $\s$ is the nontrivial Galois automorphism of
$\kay/\Q$.
Fix a fractional ideal $\frak a$ and elements $\l\in \d^{-1}\frak a/\frak a$ and $r\in \d^{-1}/\OK$,
where $\d$ is the different of $\kay/\Q$. For a positive integer $m$,  let
$\ZZ_{\Cal C}(m)=\ZZ_{\Cal C}(m;\frak a, \l, r)$
be the locus of
triples  $(E, \iota,\bbold)$,
where
\begin{enumerate}
\item $(E, \iota)$ is an object in $\mathcal C(S)$
\item $\bbold \in L(E,\iota)\frak a^{-1}$ is a special quasi-endomorphism such that
\item $r+\bbold\l \in  O_E:=\End(E/S)$, and
\item
$\deg(\bbold)= \frac{m}{N(\frak a) }$.
\end{enumerate}
Note that $L(E,\iota)\frak a^{-1}$ is a lattice in $V(E,\iota) =
L(E,\iota)\tt_\Z\Q$ and, for $x\in V(E,\iota)$, $-x^2 =
\deg(x)\,1_E$. The special cycle $\ZZ_{\Cal C}(m)$ is either empty
or is a $0$-cycle on $\Cal C$ supported in characteristic $p$ for a
prime $p$ determined by $\kay$ and $m$.  There is a corresponding
generating function (see (\ref{Cgenfun}) for the precise definition)
\beq\label{atheta.C}
\phhat_{\Cal C}(\tau; \frak a, \l, r)= \sum_m
\ZH_{\Cal C}(m,v)\,q^{\mm}, \qquad \mm = m/\Delta(\l)
\eeq
for the images under the arithmetic degree map $\CH^1(\Cal C) \overset{\degh}{\isoarrow} \R$, of the classes defined by these $0$-cycles
in the first arithmetic Chow group $\CH^1(\Cal C)$ of
$\Cal C$. Here, $q = e(\tau)$, $\tau = u+iv\in \H$, the upper half,
and the divisor $\Delta(\lambda)\mid \Delta$ is given in
(\ref{deltalambda}).  For positive $m$, $\ZH_{\Cal C}(m, v)=\ZH_{\Cal C}(m)$ is independent of $v$.  For $m\le 0$,
additional terms, depending on $v$, are defined
in section \ref{sectY2}.

These cycles and generating series are generalizations of those defined in
\cite{KRYtiny}. Indeed, when $\kay$ has prime discriminant,
the series $\phih_{\Cal C}(\tau;\OK,0,0)$ coincides, up to a constant factor, with that of \cite{KRYtiny} and
was shown there
to be a (non-holomorphic) modular form of weight $1$.

Also associated to the data 
$(\frak a,\l,r)$ is a normalized  incoherent Eisenstein series
$E^*(\tau,s;\frak a,\l,r)$ of weight $1$ and character $\chi$ for $\Gamma_0(|\Delta|)$, where $\chi$ is the quadratic character
associated to $\kay/\Q$.
The analytic continuation of such a series vanishes at $s=0$, and
our first main result, which generalizes that of \cite{KRYtiny},
describes the leading term there.

\begin{maina}
Assume that $2 \nmid \Delta$. Then
$$
 E^{*, \prime}(\tau,0;\frak a, \l, r) =  -2\, \phih_{\Cal C}(\tau;\mathfrak a, \lambda, r).
$$
In particular, $\phih(\tau;\frak a,\l,r)$ is a non-holomorphic modular form of weight $1$.
\end{maina}

The second type of arithmetic theta series are associated to the arithmetic surfaces whose generic fibers are
Shimura curves over $\Q$.
For  an indefinite quaternion algebra $B$
over $\Q$ with a fixed maximal order $O_B$, the moduli stack $\M$ of
abelian surfaces $(A,\iota)$ with $O_B$-action is an arithmetic surface over $\Spec(\Z)$.
This surface has a rich supply of divisors $\ZZ(t)$, $t\in \Z_{>0}$,  defined as the locus of triples $(A,\iota,x)$,
where $x$ is a `special' endomorphism of $A$ with square $x^2 = -t$.
Recall that such an endomorphism commutes with the given action of $O_B$ and has trace zero.
In \cite{KRYbook}, an
extensive study was made of the classes defined by the cycles $\ZZ(t)$ in the arithmetic Chow group of $\CH^1(\M)$.
More precisely, for a positive real number $v$,
there is a Green function $\Xi(t,v)$ for $\ZZ(t)$, and a resulting class $\ZH(t,v) = (\ZZ(t),\Xi(t,v))\in \CH^1(\M)$. We refer the reader to \cite{KRYbook} for more details.
One of the main results of \cite{KRYbook} is that, for $\tau= u+iv$ in the upper half plane, the generating series
\beq\label{atheta.M}
\phhat(\tau) = \sum_{t} \ZH(t,v)\ q^t,
\eeq
the arithmetic theta function of our title, is a (non-holomorphic) modular form of weight $\frac32$ and level $4 D(B)_o$ valued in $\CH^1(\M)$,
where $D(B)_o$ is the product of the odd primes at which $B$ is ramified. Here, classes for $t\le 0$ are
also included in the series\footnote{This series was denoted in \cite{KRYbook} by $\phhat_1(\tau)$; here we
omit the subscript, since the genus two generating function $\phhat_2(\tau)$ of \cite{KRYbook} will play no role in the
present paper.} .

In the present paper, we suppose that embeddings
$$\kay\overset{i}{\hookrightarrow} B\hookrightarrow M_2(\kay)$$
are given, with $i(\OK) \subset O_B$.  For an $\OK$-lattice $\L\subset \kay^2$, let
$\OL$ be the maximal order in $M_2(\kay)$ which stabilizes $\L$. For each $\OK$-lattice $\L$ such that
$O_B = \OL\cap B,$
we define a morphism
of moduli stacks
$$j_\L: \Cal C \lra \M' = \M\times_{\Spec(\Z)}\Spec(\OK),$$
corresponding to the functorial construction
of an $O_B$-module $(A,\iota)$ from a elliptic curve $(E,\iota)$
with CM by $\OK$ given by the Serre construction,
$A = \L\tt_{\OK}E$.
We assume that $\Delta$ and $D(B)$ are relatively prime so that the base change $\M'$
of $\M$ to $\Spec(\OK)$  is again regular.
Then there is a natural map
$$\CH^1(\M) \lra \CH^1(\M'),$$
and we abuse notation and also write $\phih(\tau)$ for the image of the generating series (\ref{atheta.M}) under this map.
The morphism $j_\L$ determines a map
$$j_\L^*: \CH^1(\M') \lra \CH^1(\Cal C),$$
of arithmetic Chow groups, and our main goal is to determine the arithmetic degree of the
pullback $j^*_\L(\phhat(\tau))$ of the arithmetic theta series. The
result is the following. 
\begin{mainb}
Assume that $2\nmid \Delta$ and that $\Delta$ and $D(B)$ are relatively prime. Then
$$\degh\,j^*_\L(\phih(\tau)) =\sum_{\substack{r\in \d^{-1}/\OK\\ \snass \tr(r)=0}}
\theta(\tau;r)\,\widehat{\phi}_{\Cal C}(D(B)\tau; \bar{\frak a}, \l',r),$$
where,
$$
\theta(\tau;r)= \sum_{\substack{\a\in \d^{-1}\\
\snass \tr(\a)=0\\ \snass \a\equiv r\!\! \!\mod \OK}} q^{N(\a)},
$$
is a theta series of weight $\frac12$ depending on $r$.
Here $\frak a$ is a fractional $\OK$-ideal and $\l$ is a generator for the cyclic $\OK$-module $\d^{-1}\frak a/\frak a$ determined
by the embedding of $\OK$ into $O_B$, cf. Proposition~\ref{maxorders}, and $\l'$ is a twist of $\l$, cf. Proposition~\ref{special.endos}.
\end{mainb}

Such a relation is analogous to the following simple identity for classical theta series.  Suppose that $L$ is an integral lattice in
a quadratic space $(V,Q)$ and that an orthogonal decomposition $V = V_0+V_1$ is given.
Then, the classical theta series
$$\theta(\tau,L) = \sum_{x\in L} q^{Q(x)}$$
has a factorization\begin{equation}\label{factor.theta}
\theta(\tau,L) = \sum_{r \in L^\vee/L} \theta(\tau,L_0,r_0)\,\theta(\tau,L_1,r_1),
\end{equation}
where $r= r_0+r_1$ runs over the cosets of $L$ in the dual lattice $L^\vee$, and $L_i = L\cap V_i$.
We expect that such relations will hold for the pullbacks of other arithmetic theta series and that they
will be useful in applications to special values of derivatives of $L$-functions, just as the factorization formula
(\ref{factor.theta}) plays an important role in the study of special values of $L$-function.

For example, as an application of our results, we can determine the pullback of the classes
$$\thetahat(f) = \gs{f}{\phih}_{\text{Pet}}\ \ \in\CH^1(\M)$$
associated to a newform $f$ of weight $\frac32$ on $\Gamma_0(4 D(B)_o)$  via the arithmetic theta lift, \cite{KRYbook}, Chapter IX.
We compute
\begin{align*}
\degh \, j^*_\L(\thetahat(f))& = \gs{f}{\degh\, j^*_\L(\phih)}_{\text{Pet}}\\
\nass
{}&= \sum_{\substack{r\in \d^{-1}/\OK\\ \snass \tr(r)=0}}
\gs{f}{\theta(\tau;r)\,\widehat{\phi}_{\Cal C}(D(B)\tau; \bar{\frak a}, \l',r)}_{\text{Pet}}\\
\nass
{}&=  -\frac12 \frac{\d}{\d s} \bigg(\ \sum_{\substack{r\in \d^{-1}/\OK\\ \snass \tr(r)=0}}
\gs{f}{\theta(\tau;r)\,E(D(B)\tau, s; \bar{\frak a}, \l',r)}_{\text{Pet}}\ \bigg)\bigg\vert_{s=0}.
\end{align*}
Here the second line follows from the first by Main Theorem B while the third line follows from the second by Main Theorem A.
The inner integrals in the last line are the Rankin-Selberg integrals studied by Shimura in his seminal paper \cite{shimura}
on modular forms of half integral weight, and they represent the Hecke L-function of
the corresponding newform $F$ of weight $2$.  In this way, we find that
$\degh\,j^*_\L(\widehat\theta(f))$ is proportional to $L'(1,F)\cdot a(|\Delta|,f)$,
where $a(m,f) $ is the $m$-th Fourier coefficient of $f$. We hope to give the details of this computation elsewhere.

There are still two restrictions in the present paper. First, we assume that
$\Delta$ is odd. This is due to a certain lack of information about the local
Whittaker functions, section 5, and could be removed with more calculation.
The second restriction, that $D(B)$ and $\Delta$ be relatively prime, arises from the fact that the regularity of
the base change $\M'$ of the arithmetic Shimura surface $\M$ to $\Spec(\OK)$ is lost if there are primes dividing $D(B)$ that
are ramified in $\kay$. This will result is a slight shift in the contribution of the arithmetic Hodge bundle that remains to be determined.

This paper has been in progress for a long time.  
Hidden just below the surface are elaborate relations involving
the genus theory of the field $\kay$ and its interaction with the arithmetic of the quaternion algebra $B$.
Earlier versions were disfigured by complicated explicit computations with the genera.  Thanks to the nice
idea of Ben Howard about how to use the Serre construction in this situation, an idea he introduced in his
lectures \cite{Howard} at the Morningside Center in Beijing in the summer of 2009, we were able to eliminate these calculations in the
present version.

The second author would like to thank the following institutions for their support and hospitality: Max-Planck Institute for Mathematics at Bonn, University of Toronto, The AMSS and the Morningside Center of Mathematics in Beijing, and The Tsinghua University. He did some work for this project  while visiting these institutions at various time since 2006. The first author would like to thank the University of Wisconsin, Madison, for its hospitality 
during a number for visits during that time period. 

\subsection{Notation}

We write $\A$ for the ad\'ele ring of $\Q$.
We fix an
imaginary quadratic field
$\kay =\mathbb Q(\sqrt\Delta)$ with discriminant $\Delta <0$, and let $\OK$ be its  ring of integers.
Let $\chi=\chi_{\smallkay/\mathbb Q} =(\Delta, \, )_\A$ be the quadratic Dirichet character associated to $\kay/\mathbb Q$,
and let $\partial=\sqrt\Delta \OK$ be the different.  Let $\text{\rm Cl}(\kay)$, $h_{\smallkay}$, and $w_{\smallkay}$ be the ideal class group,
the class number and the number of root of unity in $\kay$ respectively.

For a fractional ideal $\mathfrak a$ and an element $\l\in \d^{-1}\frak a/\frak a$, let
$\d_\l\mid \d$ be the divisor of $\d$ such that $\l$ generates $\d_\l^{-1}\frak a/\frak a$.
We write
\begin{equation}\label{deltalambda}
\Dl = N(\d_\l).
\end{equation}

Let $\psi=\prod \psi_p$ be the `canonical' unramified additive character of $\mathbb Q \backslash \A$ such that $\psi_\infty(x) =e(x) =e^{2 \pi i x}$.

\medskip
\medskip

\centerline{\Large\bf Part I}

\section{\bf  The moduli problem and the incoherent Eisenstein
series} \label{sectY2}

\subsection{The moduli problem and special cycles}
For our fixed imaginary quadratic field $\kay$ with ring of integers $\OK$,
let $\mathcal C$ be the moduli stack over $\OK$  of CM elliptic curves $(E, \iota)$ as in \cite{KRYtiny},
to which we refer the reader for more details.
For an
$\OK$-scheme $S$ and a CM elliptic curve $(E, \iota) \in \mathcal C(S)$, let
\begin{equation} \label{eqY2.1}
L(E, \iota) = \{  j \in \End(E/S) \mid\,   j \circ \iota(a) =
\iota(\bar a) \circ j, \quad a \in  \OK \}
\end{equation}
be the lattice of special endomorphisms, and let $V(E, \iota) =
L(E, \iota) \otimes \mathbb Q$. This $\Q$-vector space is equipped with a canonical
quadratic form $Q(j) =\deg j$.

We now introduce special cycles that are a slight generalization of those defined in \cite{KRYtiny}.

\begin{defn}\label{ZZdef}
Fix a rational  number  $m \ge 1$,  a fractional ideal $\frak a$, an element $\lambda \in \partial^{-1} \mathfrak a/\mathfrak a$,
and  an element $r \in \partial^{-1}/\OK$.
Let $\ZZ(m) = \mathcal Z(m; \mathfrak a, \lambda, r)$ be the fibered category over $\text{\rm Sch}/\OK$ that associates to an
$\OK$-scheme $S$ the category $\mathcal Z(m; \mathfrak a, \lambda, r)(S)$
whose objects are triples  $(E, \iota,\bbold)$,
where
\begin{enumerate}
\item $(E, \iota)$ is an object in $\mathcal C(S)$
\item $\bbold \in L(E,\iota)\frak a^{-1}$,
\item $r+\bbold\l \in  O_E:=\End(E/S)$,
\item
$\deg(\bbold)= \frac{m}{N(\frak a) }$.
\end{enumerate}
The morphisms in the category are $\OK$-linear isomorphisms
$\phi: (E,\iota) \rightarrow (E',\iota')$
of elliptic schemes over $S$ such that $\phi^*\bbold' = \bbold$.
\end{defn}

\begin{rem}
(i) Here $\bbold$ is an element of $\End^0(E/S) = \End(E/S)\tt_\Z\Q$ and, in condition (3),  we choose any
representative of $r$ in $\d^{-1}$ and of $\l$ in $\d^{-1}\frak a$.  \hfb
(ii)  The cycle $\ZZ(m;\frak a, \l, r)$
coincides with that defined in \cite{KRYtiny} in the case
$\frak a = \OK$ with $r=\l=0$. \hfb
(iii) This particular generalization of the definition in \cite{KRYtiny} is motivated by
the results about the pullback for cycles on Shimura curves that will be obtained in Part II of this paper. \hfb
(iv) For any $(E,\iota)$ in $\Cal C(S)$, the Serre construction gives rise to an elliptic scheme
$E_\frak a:=\frak a \tt_{\OK}E$ over $S$ with CM by $\OK$. Then, the element $\bbold$ can be viewed as
an $\OK$-anti-linear homomorphism
$$\bbold \in \Hom((E_{\frak a},\iota), (E,\bar\iota)),$$
where $(E,\bar\iota)$ is the elliptic scheme $E$ with $\OK$-action given by $\bar\iota(a) = \iota(\bar a)$.
\end{rem}
The same argument as in  \cite[Section 5]{KRYtiny} shows that this moduli problem  is represented
 by a stack, still denoted by $\mathcal Z(m)$,
 whose coarse moduli scheme $\und{\ZZ(m)}$ is
 a  finite Artinian $\OK$-scheme
 which is only supported on primes non-split in $\kay$.
 It is clear that $\mathcal Z(m)$ is empty unless $m$ is a positive integer
 and  $ \partial_r \subset \partial_\lambda$.   Here $\d_r$  (resp. $\d_\l$) is the divisor of $\d$ such that $r$ (resp. $\l$)
 generates $\d_r^{-1}/\OK$ ( resp. $\d_\l^{-1}\frak a/\frak a$).

 The forgetful functor defines a morphism
 of stacks
 $$
 \mathrm{pr}:  \mathcal Z(m) \rightarrow \mathcal C,  \quad  (E, \iota, \bbold) \mapsto (E, \iota).
 $$
 Recall from section 5 of \cite{KRYtiny} that  the
 Arakelov
degree of $\mathcal Z(m)$ is defined to be
\begin{align}
\degh\, \mathcal Z(m) &= \sum_{\mathfrak p}  \log
N(\mathfrak p) \sum_{x \in \mathcal Z(m) (\overline{\kappa(\mathfrak
p)}) } \frac{1}{|\Aut_{\mathcal C}(\mathrm{pr}(x))|} \lg(x)
\notag
\\
 &=\frac{1}{w_\smallkay} \sum_{\mathfrak p}  \log
N(\mathfrak p) \sum_{x \in \mathcal Z(m) (\overline{\kappa(\mathfrak p)}) }
 \lg(x).
\end{align}
Here $\mathfrak p$ runs over the primes of $\kay$,
$\kappa(\mathfrak p)$ is the residue field of $\kay $ at $\mathfrak p$,  and $\lg(x)$ is the length of the local ring
$\OO_{\ZZ(m), x}$:
\begin{equation}
\lg(x) = \hbox{length of } \OO_{\ZZ(m), x} =\hbox{length of }
\widehat{\OO}_{\ZZ(m), x}.
\end{equation}

\begin{rem}  \label{remY2.1} We choose
a preimage $\tilde \l$ of $\l$ in $\d^{-1}\frak a$ and $\tilde r$ of $r$ in $\d^{-1}$.
For $\ell\mid \Delta$, denote by $\l_\ell$ (resp. $r_\ell$) the image of $\tilde\l$ (resp. $\tilde r$) in
$\d^{-1}\frak a\tt\Z_\ell$ (resp. $\d^{-1}\tt\Z_\ell$).
Then the condition $r+\bbold\l
\in  O_E:=\End(E)$ is the same as
$$
r_\ell + \beta \lambda_\ell \in O_{E, \ell}=O_E \otimes \mathbb Z_\ell,  \qquad \text{for all }
\ell\mid\Delta.
$$
We write $\lbold$ (resp. $\rbold$) for the adele with components $\l_\ell$ (resp. $r_\ell$)
for $\ell \mid \Delta$ and $0$ elsewhere.
\end{rem}

 The following scaling relation is immediate.
\begin{lem} For $\alpha \in \kay^\times$,  there is an isomorphism
$$
\mathcal Z(m; \mathfrak a, \lambda, r) \isoarrow \mathcal Z(m; \alpha
\mathfrak a, \alpha \lambda, r),  \qquad (E, \iota, \bbold) \mapsto
(E,  \iota,  \bbold\, \iota(\alpha)^{-1} ).
$$
\end{lem}

\newcommand{\Top}{\mathrm{top}}

For a negative integer $m <0$, we define an `arithmetic cycle for
$\mathcal C$ supported at $\infty$'  as follows\footnote{This
construction is based on Gross's observation \cite[Page
383]{KRYtiny} that $E^\top$ should be viewed as the archimedean
analogue of a supersingular elliptic curve, since its endomorphism
algebra $\End(E^\Top) \simeq M_2(\Z)$ is a maximal order in the
quaternion algebra $\mathbb B$ ramified at $\infty$ and $\infty$,
i.e., everywhere unramified!}.  For a CM  elliptic curve $(E,
\iota)$ over $\mathbb C$, let $E^{\Top}$ be the underlying real
torus. As in (\ref{eqY2.1}), we define the lattice of special
endomorphisms
$$
L(E^{\Top}, \iota) = \{  j \in \End(E^{\Top}):\,   j \circ
\iota(a) = \iota(\bar a) \circ j, \quad a \in  \OK \},
$$
equipped with a quadratic form $Q(j)= -j^2$. It is easy to check
that $L(E^{\Top}, \iota)$ is a projective $\OK$-module of rank $1$
and that the quadratic form $Q$ is negative definite.
\begin{defn}\label{def.ztop}
Let $\ZZ^{\top}(m)$ be the category of triples $(E, \iota,
\bbold)$ where
\begin{enumerate}
\item $(E, \iota) \in \mathcal C (\mathbb C)$
\item $\bbold \in L(E^{\Top},\iota)\frak a^{-1}$,
\item $r+\bbold\l \in  O_{E^{\Top}}:=\End(E^{\Top})$,
\item
$Q(\bbold)= \frac{m}{N(\frak a) }$.
\end{enumerate}
\end{defn}
The forgetful functor defines a map $\pr: \ZZ^\top(m) \rightarrow \Cal C(\C)$ with finite fibers.
We denote the set of isomorphism classes of objects in $\ZZ^\top(m)$ by $\und{\ZZ^\top(m)}$.

For a negative integer $m<0$ and a parameter $v\in \R^\times_+$, we define a real valued
function on $\und{\Cal C}(\C)$ by
\begin{equation}\label{arch.defZZ}
\mathcal Z(m, v)(E,\iota)=\mathcal Z(m, v; \mathfrak a, \lambda, r)(E,\iota)
=\sum_{\substack{x \in \und{\ZZ^\top(m)}\\ \snass \pr(x) = (E,\iota)}} \lg(x, v),
\end{equation}
where
\begin{equation}\label{arch.length}
\lg(x, v) =  \beta_1(4 \pi |\mm| v)
\end{equation}
is `length' of the $x$.  Here, $\Dl$ is given by (\ref{deltalambda}) and
$$
\beta_1(a)=\int_1^\infty u^{-1} e^{-ua} du
$$
is the partial Gamma function or exponential integral.

As in section 6 of \cite{KRYtiny}, we view $\ZZ(m,v)$ as an
Arakelov divisor on $\und{\Cal C}$. Its Arakelov degree is
\begin{equation}
\degh\, \mathcal Z(m, v) = \sum_{ x \in \und{\ZZ^\top(m)}}
\frac{1}{|\Aut_{\mathcal C}(\mathrm{pr}(x))|} \lg(x, v)
 = \frac{1}{w_{\smallkay}}\,\beta_1(4 \pi |\mm| v)\, |\und{\ZZ^\top(m)}|.
 \end{equation}

Note that, for $m<0$, the coefficient of $q^m$ in the generating function in \cite{KRYtiny}
 is $\b_1(4\pi v|m|)\,w_{\smallkay}\,\rho(-m)$ where $w_{\smallkay}=2$.
 For comparison with this case, we note the following.
\begin{lem} Suppose that $\Delta$ is a prime. Then, for
$\frak a = \OK$ and $r=\l=0$,
$$|\und{\ZZ^\top(m)}| = w_{\smallkay}\,\rho(-m),$$
where $\rho(n)$ is the number of integral ideals of $\kay$ of norm $n$.
\end{lem}
\begin{proof}  Suppose that $E = \C/\frak b$, and let $j_\s:\C/\frak b \isoarrow \C/\bar{\frak b}$ be the topological isomorphism
given by complex conjugation. Then, for $\bbold\in L(E^\top,\iota)$ with $Q(\bbold)=m$,
map
$ j_\s\circ\bbold: \C/\frak b \rightarrow \C/\bar{\frak b}$
is $\OK$-linear and holomorphic and thus is given by multiplication by some $\b\in \bar{\frak b} \frak b^{-1}$ with
$N(\b) = -Q(\bbold)=-m$. The integral ideal
$\b \frak b \bar{\frak b}^{-1} \subset \OK$ has norm $-m$ and lies in the ideal class $[\frak b \bar{\frak b}^{-1}]= [\frak b]^2$.
As $\frak b$ varies over representatives for the ideal classes, so does $[\frak b]^2$, since the genus group is trivial for a prime discriminant,
hence the claim.
\end{proof}

Finally, we define $\mathcal Z(0, v)$ to be the Arakelov divisor
supported at $\infty$ with degree
\begin{equation}
\degh\, \mathcal Z(0, v) =\begin{cases}
 - \Lambda'(1, \chi) - \frac12\,\Lambda(1, \chi)\log(v) &\hbox{if }  r \in \OK,
   \\
    0 &\hbox{if }  r \notin \OK,
     \end{cases}
\end{equation}
 where\footnote{Note that we have added the factor $|\Delta|^{\frac{s}2}$ here as compared with the
 convention used in \cite{KRYtiny}.}
\begin{equation}\label{Lschi}
\Lambda(s, \chi) = |\Delta|^{\frac{s}2} L_\infty(s, \chi) L(s,
\chi)
\end{equation}
is the complete $L$-function of $\chi$. Here
 $$
 L_\infty(s, \chi) =\pi^{-\frac{s+1}2} \Gamma(\frac{s+1}2).
 $$
Note that the constant term of the generating function in \cite{KRYtiny} is
\begin{equation}
-w_{\smallkay}\,\big(\, \L'(1,\chi) + \frac12\,\L(1,\chi)\,\log(v)\,\big ).
\end{equation}

With these definitions and for $\tau=u+iv\in \H$, we define the generating function
\begin{equation}\label{Cgenfun}
\phih(\tau;\mathfrak a, \lambda, r)
 = \degh\,\ZZ(0,v) + \sum_{m \in \mathbb Z_{< 0}} \widehat{\deg}\, \mathcal Z(m, v)
\, q^\mm + \sum_{m \in \mathbb Z_{>0}} \widehat{\deg}\, \mathcal Z(m)
\, q^\mm,
 \end{equation}
 where we have suppressed the dependence on $(\frak a, \l, r)$ on the right side.

The purpose of Part I is to prove that  the generating function
$\phih(\tau;\mathfrak a, \lambda, r)$ is a (non-holomorphic)
modular form of weight $1$ for some congruence subgroup of
$\SL_2(\mathbb Z)$, and to identify it with the central derivative
of an incoherent Eisenstein series. This is a generalization of the main result
in \cite{KRYtiny}.
Indeed, in the case $|\Delta|=q$, $\frak a= \OK$, $\l=r=0$, our
$\phih(\tau,\OK, 0,0)$ is just $1/w_{\smallkay}$ times the generating function in \cite{KRYtiny}.
The new idea is to use the Siegel-Weil formula to avoid some lengthy explicit calculations.

\subsection{Quadratic spaces and Eisenstein series} We briefly review some basic facts about the
Weil representation and
 Eisenstein series for use in the rest of the paper, referring to \cite{KYeis} for more information.
Let $\chi$ be a quadratic character of $\A^\times/\Q^\times$.

For a quadratic space $(V,Q)$ of dimension $n$ over $\Q$
with\footnote{Recall that $\chi_{V}(x)=((-1)^{\frac{n(n-1)}2} \det
V, x)_\A$.} $\chi_V=\chi$, we obtain a collection
$V=\{V_p\}_{p\le\infty}$ of local quadratic spaces, $V_p =
V\tt_\Q\Q_p$, and  $V(\A) = \prod_{p\le \infty}' V_p$ is the
restricted product with respect to the collection of compact open
subgroups $L_p=L\tt_\Z\Z_p$, for any lattice $L$ in $V$.

More generally, we can consider a collection $\Cal V = \{\Cal V_p\}$ of local quadratic spaces
with $\chi_{\Cal V_p}= \chi_p$ for all $p$ and agreeing with a coherent collection at almost all places\footnote{This means that
there exists a global quadratic space $V$ and isomorphisms $\phi_p: V_p\isoarrow \Cal V_p$ for almost all $p$.
If $(V',\{\phi_p'\})$ is another such collection, we require that $\phi_p^{-1}\circ \phi_p':V'_p \isoarrow V_p$
carry $L'_p$ to $L_p$ for almost all $p$ for some lattices $L$ in $V$ and $L'$ in $V'$.}.
The restricted product $\Cal V(\A) = \prod_{p\le \infty}'\Cal V_p$ is then defined, and the collection $\Cal V$
is called coherent (resp. incoherent) if the global Hasse invariant
$$\e(\Cal V) = \prod_{p\le \infty} \e(\Cal V_p)$$
is $+1$ (resp. $-1$).  In the coherent case, the collection $\Cal V$ arises by localization from a global quadratic space,
unique up to isomorphism. In the incoherent case, there are infinitely many such global spaces at `distance one'
from $\Cal V$.

For a collection $\Cal V$ and our fixed additive character $\psi$, there
is a Weil representation $\omega_{\Cal V, \psi}$ of $\SL_2(\A)$ on $S(\Cal V(\A)) =\otimes'_p S(\Cal V_p)$. It is given by
\begin{align} \label{eqWeil}
 \omega_{\Cal V, \psi}(n(b)) \ph(x) &= \psi(b\, Q_\A(x)) \ph(x), \quad b \in
 \A,  \notag
 \\
 \nass
 \snass
  \omega_{\Cal V, \psi}(m(a)) \ph(x) &=\chi(a) |a|^m \ph(xa), \quad a \in \A^\times,
  \\
  \nass
   \omega_{\Cal V, \psi}(w^{-1}) \ph(x) &= \gamma(\Cal V) \int_{\Cal V(\A)} \ph(y) \psi(-(x, y))
   dy. \notag
\end{align}
Here $\gamma(\Cal V) =\prod_p \gamma(\Cal V_p)$, where $\gamma(\Cal V_p)$ is the local Weil index \cite{Rao},
\cite{Kusplit}, an
$8$-th root of unity, and we write
$$
n(b) =\kzxz {1} {b} {0} {1},  \quad m(a) = \kzxz {a} {0} {0}
{a^{-1}}, \quad w=\kzxz {0} {1} {-1} {0}.
$$
Let $I(s,\chi)$ be the representation of $\SL_2(\A)$ induced from the character of the Borel subgroup $B$ given by
$n(b)m(a) \mapsto \chi(a)\,|a|^{s+1}$. The induction is normalized in the standard way so that $\Re(s)=0$ is the unitary axis.
 There is an $\SL_2(\A)$-equivariant  map
$$
\P:  S(\Cal V(\A)) \rightarrow I(s_0, \chi), \quad \P_{\ph}(g) = \omega_{\Cal V, \psi}(g) \ph(0),
$$
where $s_0 =\frac{n-2}2$.  We denote the image of this map by $R(\Cal V)$, when $\Cal V$ is
incoherent, and by $R(V)$, when $\Cal V$ is coherent with associated global quadratic space $V$.
Notice that for $n=2$, the representation $I(0,\chi)$ is unitarizable, the $R(\Cal V)$'s and $R(V)$'s are irreducible,
and there is a decomposition
$$
I(0, \chi) = (\oplus_{ \e(\Cal V)=+1}R(V) )\oplus (\oplus_{ \e(\Cal
V)=-1} R(\Cal V).
$$

For $\ph \in S(\Cal V(\A))$, let $\Phi(s)=\Phi_\ph(s) \in I(s, \chi)$ be the associated standard\footnote{This means that the
restriction of $\P(s)$ to the maximal compact subgroup $K = K_\infty \,\prod_p\SL_2(\Z_p)$ is independent of $s$.}
section with $\Phi(g, s_0) =\P_\ph(g)$, and let
$$
E(g, s, \ph) = \sum_{\gamma \in B(\Q) \backslash \SL_2(\mathbb Q)} \Phi(g, s).
$$
be the associated Eisenstein series.
It is absolutely convergent for $\Re \,s>1$,
has meromorphic  analytic continuation in $s$, and is holomorphic at $s=s_0$.
The Eisenstein series attached to $\ph\in S(\Cal V(\A))$ will be called
coherent (resp. incoherent) if $\Cal V$ is coherent (resp. incoherent).
When $\Cal V_\infty$ is positive definite, and  $\ph_\infty(x) =e^{- 2 \pi Q_\infty(x)}$,
$$
E(\tau, s,  \ph) = v^{-\frac{n}4} E(g_\tau, s,  \ph)
$$
is a (non-holomorphic) modular form of weight $\frac{n}2$ for some congruence
subgroup of $\SL_2(\mathbb Z)$, where, for $\tau = u + i v \in  \mathbb H$, the upper half-plane,  $g_\tau =n(u) m(\sqrt v)$.
The Eisenstein series  has Fourier expansion
$$
E(g, s, \ph) = \sum_{ m \in  \mathbb Q} E_m(g, s, \ph),
$$
and, for $m \ne 0$,
$$
E_m(g,s, \ph) = \prod_p W_{m, p}(g_p, s, \ph_p).
$$
Here, for any $m\in \Q$,
$$
W_{m, p}(g_p, s,  \ph_p) =\int_{\mathbb Q_p} \Phi_p(w^{-1} n(b)
g_p, s)\, \psi_p(-mb)\, db
$$
is the local Whittaker function. The constant term is given by
$$
E_0(g, s, \ph) =\Phi(g, s) + \prod_p W_{0, p}(g_p,s, \ph_p).
$$
Finally, when $\dim V=2$ and $\chi =\chi_{\smallkay/\mathbb Q}$, we define the normalized
Eisenstein series
$$
E^*(\tau, s, \ph) = \Lambda(s+1, \chi)\, E(\tau, s, \ph),
$$
and normalized local Whittaker functions
$$
W_{m,p}^*(g_p, s, \ph_p) = |\Delta|_p^{-\frac{s+1}2} L_p(s+1, \chi)\, W_{m, p}(g_p, s, \ph_p),
$$
for $p < \infty$, and
$$
W_{m, \infty}^*(\tau, s, \ph_\infty) = v^{-\frac{1}2}
L_\infty(s+1, \chi) \, W_{m, \infty}(g_\tau, s, \ph_\infty).
$$
For a finite prime $p$, we will frequently write
\begin{equation}\label{forshort}
W_{m,p}^*(s, \ph_p)=W_{m,p}^*(1, s, \ph_p),
\end{equation}
when $g_p=1$.

Recall that the values at $s=0$ of coherent Eisenstein series are given in terms of binary theta series,
a classical version of the Siegel-Weil formula,
while all incoherent Eisenstein series vanish at this point.

\subsection{The central derivative of an incoherent Eisenstein series}

We fix data $\frak a$, $\l\in \d^{-1}\frak a/\frak a$ and $r\in \d^{-1}/\OK$ as before.
Let $V=\kay$ with quadratic form $Q(x) = -N(\d_\l^{-1}\frak a)\,N(x)$.  Let $\Cal V$
be the incoherent collection with $\Cal V_p=V_p$ for all finite primes and
with $\Cal V_\infty = \kay_\R$, $Q_\infty(x) = N(\d_\l^{-1}\frak a)\,N(x)$.
Recall that $N(\d_\l)= \Delta(\l)$, (\ref{deltalambda}).

Following \cite{kudla.annals}, for a non-zero rational number $m$, let
$\Diff(\mathcal V, m)$ be the set of
primes $p$ where $\mathcal V_p$ does not represent $m$.
It is clear that $p \in  \Diff(\mathcal V,m)$ if and only if
\begin{equation}
 \chi_p(-m)  =
 \begin{cases}-1&\text{if $p<\infty$}\\
 1&\text{if $p=\infty$,}
\end{cases}
\end{equation}
where $\chi_p(x) = (\Delta,x)_p$.  In particular, $|\Diff(\Cal V,m)|$ is odd.

Let $\ph =\ph_{\mathfrak a, \lambda, r}=\tt_\ell \ph_\ell \in
S(\mathcal V(\A))$, where
  \begin{equation}\label{mainph}
  \ph_\ell(x) = \begin{cases}
     \cha(\mathfrak a_\ell^{-1})(x) &\hbox{if } \ell \nmid \Delta \infty,
     \\
      \cha(\mathfrak a_\ell^{-1})(x) \cdot\cha(-\bar r_\ell+ O_{\smallkay,\ell})(x \lambda_\ell) &\hbox{if } \ell |\Delta ,
      \\
      e^{-2 \pi  Q_\infty(x)} &\hbox{if }  \ell =\infty.
      \end{cases}
      \end{equation}
 Here $r_\ell$ is the image of $r$ in $(\d^{-1}/\OK)\tt\Z_\ell$.
Let  $E^*(\tau,s;\ph)= E^*(\tau,s;\frak a, \l,r)$ be the associated normalized incoherent
Eisenstein series.

Our first main result is the following.
 \begin{theo} \label{maintheo1}  Assume $2 \nmid \Delta$. Then
   $$
 E^{*, \prime}(\tau,0;\frak a, \l, r) =  -2\, \phih(\tau;\mathfrak a, \lambda, r).
   $$
\end{theo}

\begin{rem} The assumption $2 \nmid \Delta$ is technical and is
only made to simplify Proposition \ref{Whittaker} below.
\end{rem}

We will prove the  theorem by comparing Fourier
coefficients, i.e., by proving that, for each  integer $\mm = m/\Delta(\l)$,
\begin{equation}
E_{\mm}^{*, \prime}(\tau, 0;\frak a,\l,r)
= -2\, \widehat{\deg}\, \mathcal Z(m, v)\, q^{\mm} .
\end{equation}

For $m <0$
we will sketch a proof very similar to the case $m >0$ in Section
\ref{sectY6}. The case $m=0$ follows from the definition of $\widehat{\ZZ}(0,v)$
and the computation of the constant term of the Eisenstein series,
and is left to the reader.  The proof of the case $m>0$ can be divided into three
parts. It is easy to show that $\mathcal Z(m)$ can only be supported
at primes $p$ non-split in $\kay$.  First, in Section
\ref{sectCounting}, we  compute $\und{\mathcal Z(m)} (\bar{\mathbb
F}_p)$, for such a $p$. Using the fact that the class group
$\CL(\kay)$ acts  transitively on $\mathcal C(\bar{\mathbb F}_p)$,
we can write this quantity as a theta integral, thereby avoiding a
lengthy calculation involving genus theory that disfigured an
earlier draft of this paper.  This device  is inspired by Ben
Howard's lectures at the Morningside Center of Mathematics in
Beijing, \cite{Howard}, in summer 2009. Using the Siegel-Weil
formula, we see that $\und{\mathcal Z(m)} (\bar{\mathbb F}_p)$ is
equal to the $\mm$-th Fourier coefficient of a coherent
Eisenstein series $E^*(\tau, 0, \ph^{(p)})$ (Theorem
\ref{counting}), which comes from a coherent collection $V^{(p)}$
differing from $\mathcal V$ exactly at $p$. In Section
\ref{sectLength}, we use Gross's canonical lifting to compute the
length $\lg(x)$ for a point $x =(E, \iota, \bbold) \in  \mathcal
Z(m) (\bar{\mathbb F}_p)$, and proved that it  depends  only on $m$.
Therefore, one has (Theorem \ref{theodegree} )
$$
\widehat{\deg}\,\mathcal Z(m)\, q^{\mm}= \frac{1}4 \,c_p(m)\, \log p \cdot E_{\mm}^*(\tau, 0; \ph^{(p)})
$$
for some number $c_p(m)$  (basically $\lg (x))$ depending only on
$m$. In Section  \ref{sectWhittaker}, we first observe that the
incoherent Eisenstein series  $E(\tau, 0; \ph)$ and the coherent
Eisenstein series $E(\tau, 0; \ph^{(p)})$ are closely
related---a general  phenomenon (Proposition \ref{propY5.1}):
$$
 W_{\mm, p}^*(0, \ph^{(p)}_p)\, E_{\mm}^{*, \prime}(\tau, 0;\ph)
   =W_{\mm, p}^{*, \prime}(0, \ph_p)\, E_{\mm}^*(\tau, 0;
   \ph^{(p)}).
$$
Now all we need is to prove
$$
\frac{W_{\mm, p}^{*, \prime}(0, \ph_p)}{
W_{\mm, p}^*(0, \ph^{(p)}_p)} =-\frac{1}2
c_p(m) \log p,
$$
which we do in Section  \ref{sectWhittaker} by explicit calculation.

\subsection{Variants}

 Many variations are possible.
For example, on the geometric side, Bruinier and the second author, \cite{BY},
consider a slight different moduli problem with a different motivation as follows.  Let $\mathfrak b$ be a fractional ideal of $\kay$, and let $L=\mathfrak b$ be equipped with the  integral quadratic form $Q(x) =-\frac{N(x)}{N(\mathfrak b)}$, so that its dual lattice is $L'= \partial^{-1} \mathfrak b$. For $\mu \in L'/L$ and $m \in \mathbb Q_{>0}$, we consider the moduli stack $\mathcal Z(m, \mu; \mathfrak b)$  over $\OK$
representing triples $(E, \iota, \bbold)$ such that
\begin{enumerate}
\item
$(E, \iota)$ is a CM elliptic curve as above,
\item
$\bbold \in L(E, \iota) \partial^{-1} \mathfrak b = L(E, \iota) \otimes_{\mathbb Z} L'$,
\item
$ \mu +\bbold  \in  O_E \mathfrak b$.
\item
$\deg \bbold = m N(\mathfrak b)$.
\end{enumerate}

For $2 \nmid \Delta$, choose a generator  $\lambda$ of   $ \mathfrak b^{-1}/\mathfrak b^{-1} \partial$.   Then  it is easy to see
that
\begin{equation}
\mathcal Z(m, \mu; \mathfrak b) = \mathcal Z(m|\Delta|; \partial \mathfrak b^{-1}, \lambda, \lambda \mu ).
\end{equation}
So Theorem \ref{maintheo1}  immediately gives the following result that was used in \cite{BY}, Theorem~6.4.

\begin{cor} Let $V=\kay$ with quadratic form $Q(x) = -N(\mathfrak b)^{-1}N(x)$ as above, and assume that $2 \nmid \Delta$.
Let  $\mathcal V^{\mathfrak b}$ be the incoherent collection with $\Cal V^{\frak b}_\ell = V_\ell$ at all finite primes $\ell$
and with $\Cal V^{\frak b}_\infty = \kay_{\,\R}$ with $Q_\infty(x) = N(\frak b)^{-1} N(x)$.
Let $L=\frak b$ and view $\hat L$ as a lattice in $\mathcal V^{\mathfrak b}(\A_f)$.
For $\mu \in L'/L$, let
     $$\phi^\mu = \cha(\mu +\hat{L}) e^{-2 \pi \frac{x \bar x}{N(\mathfrak b)}} \in S(\mathcal V^{\mathfrak b}).$$ Then
     $$
    \varphi_{\mu}(\tau):= \sum_{m \in \frac{1}{|\Delta|} \mathbb Z} \mathcal Z(m, \mu, v; \mathfrak b) q^m = -\frac{1}2 E_m^{*, \prime}(\tau, 0, \phi^\mu).
     $$
     Here $\mathcal Z(m, \mu, v; \mathfrak b) =\mathcal Z(m, \mu; \mathfrak b)$ for $m >0$ as above, and is defined to be $\mathcal Z(m|\Delta, v|; \partial \mathfrak b^{-1}, \lambda, \lambda \mu )$ for $m \le 0$.
\end{cor}

  In \cite{BY}, this theorem is derived from a result in an early version of this paper.
  This theorem implies that
  $$
  \varphi_L(\tau)= \sum_{ \mu \in L'/L} \varphi_\mu(\tau) \cha(\mu +\hat{L})
    $$ is a vector-valued modular form for $\SL_2(\mathbb Z)$  with respect to the Weil representation \cite{BY}.

\section{\bf Counting}
\label{sectCounting}
We fix data $\frak a$, $\l\in \d^{-1}\frak a/\frak a$, $r\in \d^{-1}/\OK$ and $m\in \Q^\times_{>0}$ as before, and we recall that
the cycle $\ZZ(m) = \ZZ(m;\frak a, \l ,r)$ has support in the non-split primes of $\OK$.
In this section, we express the quantity $|\und{\ZZ(m)}(\bF)|$ as a Fourier coefficient of a coherent Eisenstein
series.

We fix a non-split prime $p$ and choose a supersingular CM elliptic
curve $(E,\iota)$ over $\bF$. Then the space of special endomorphisms
$V^E=V(E, \iota)$ is a one dimensional $\kay$-vector space with a positive definite
$\mathbb Q$-valued quadratic form $Q(j) =\deg j$. Define  $\ph^E =\tt_\ell
\ph^E_\ell \in S(V^E(\A))$ with local components
\begin{equation} \label{propY4.2}
\ph^E_\ell(x) =\begin{cases}
   \cha(L(E, \iota) \mathfrak a_\ell^{-1})(x) &\hbox{if } \ell \nmid \Delta \infty,
     \\
     \cha(L(E, \iota) \mathfrak a_\ell^{-1})(x)\cdot \cha(O_{E, \ell})(r_\ell+x \lambda_\ell)  &\hbox{if } \ell | \Delta,
       \\
    e^{-2 \pi \deg x} &\hbox{if } \ell  =\infty.
     \end{cases}
     \end{equation}
Here $r_\ell$ and $\l_\ell$ are defined in Remark~\ref{remY2.1}.
Thus, writing $\ph^E_f$ for the finite part of $\ph^E$, we have $\ph^E_f(\b)\ne 0$ for $\b\in V^E(\Q) = V(E,\iota)$
precisely when $\b$ satisfies conditions (2) and (3) in the definition of $\ZZ(m)$.
\begin{prop} \label{propY3.2}
$$
|\und{\mathcal Z(m)}(\bar{\mathbb F}_p)|\,
q^{\frac{m}{N(\mathfrak a)}}=\frac{w_{\smallkay}}{4} \, E_{\frac{m}{N\mathfrak a}}^*(\tau, 0; \ph^{E}).
$$
\end{prop}
 \begin{proof}  Let $H=\SO(V^E)$, so that $H(\Q) \cong \kay^1$, and let $dh$ be the Tamagawa measure on $H(\A)$, so that
 $\vol(H(\mathbb Q) \backslash H(\A))=2$. Let
 $$
 \theta(\tau, h,  \ph^E) = v^{-\frac{1}2} \sum_{x \in V^E(\Q)} \omega_{V^E, \psi}(g_\tau) \ph^E_\infty(h_\infty^{-1}x)\ph^E_f(h_f^{-1}x)
 $$
 be the theta kernel, and let
 $$
 I(\tau, \ph^E) = \int_{H(\mathbb Q) \backslash H(\A)} \theta(\tau, h, \ph^E)\, dh
 $$
 be the associated theta integral. Then  the Siegel-Weil formula asserts
 \cite[Theorem 4.1]{KuIntegral}
 $$
 I(\tau, \ph^E) = E(\tau, 0; \ph^E).
 $$
 On the other hand, a simple calculation gives
 $$
 \theta(\tau, h,  \ph^E) = \sum_{x \in V^E(\Q)} \ph^E_f(h_f^{-1}x)\, q^{\deg x},
 $$
 so that
$$
\theta_m(\tau, h, \ph^E) = q^m \,\sum_{x \in V^E(\Q), \deg x =m} \ph^E_f(h^{-1}x),
$$
and
\begin{align*}
I_m(\tau, \ph^E) &=\int_{H(\mathbb Q) \backslash H(\A)} \theta_m(\tau, h, \ph^E)\, dh
   \\
    &= \vol(H(\mathbb R))\, q^m\, \int_{H(\mathbb Q) \backslash H(\A_f)} \sum_{x \in V^E(\Q), \deg x =m} \ph^E_f(h_f^{-1}x)\, dh
   \end{align*}
Let $\pi(t) = t \bar t^{-1}$ be the map from $\kay^\times$ to $H(\mathbb Q)$. It induces an isomorphism
   $$
 \CL(\kay)=  \kay^\times\backslash \kay_{\A_f}^\times/\widehat{O}_{\smallkay}^\times \cong H(\mathbb Q) \backslash H(\A_f)/\pi(\widehat{O}_{\smallkay}^\times).
   $$
Recall \cite[Section 5]{KRYtiny}  that $\CL(\kay)$ acts simply transitively on
$\mathcal C(\bar{\mathbb F}_p)$ via the Serre
construction. For a finite id\`ele $t\in \kay^\times_{\A_f}$, we write
$
t.E  = (t)\tt_{\OK}E,
$
where $(t)$ is the ideal `generated by' $t$.  For any $t \in
\widehat{O}_{\smallkay}^\times$, one can check that
$$
\ph^E_f(\pi(t)x) = \ph^E_f(x),
$$
and thus we have
\begin{align*}
&I_m(\tau, \ph^E)\, q^{-m}\\
{}  &= \vol(H(\mathbb R)) \frac{\vol(\pi(\widehat{O}_\smallkay^\times))}{|\kay^1 \cap \pi(\widehat{O}_\smallkay^\times)|}
\sum_{x \in V^E(\Q), \deg x =m}\  \sum_{h \in H(\mathbb Q) \backslash H(\A_f) /\pi(\widehat{O}_\smallkay^\times)} \ph^E_f(h^{-1}x)
\\
 &= C
    \sum_{x \in V^E(\Q), \deg x =m}\ \sum_{[t] \in \CL(\smallkay) } \ph^E_f(\pi(t)^{-1} x)
\end{align*}
for some constant $C$.
Now $\ph^E_f(\pi(t)^{-1} x) =1$ or $0$, and is  $1$ if and only if
$$
\pi(t)^{-1} x \in L(E, \iota) \mathfrak a^{-1},  \qquad r + \pi(t)^{-1} x \in  O_E,
$$
or, equivalently,
$$
x \in  \pi(t) L(E, \iota) \mathfrak a^{-1} = L(t.E, \iota)\mathfrak a^{-1}, \qquad r +\beta \lambda \in t O_{E} t^{-1} =O_{t.E}.
$$
 Thus  we have
\begin{align*}
I_{\frac{m}{N(\mathfrak a)}}(\tau, \ph^E)\, q^{-\frac{m}{N(\mathfrak a)}}
 &= C \sum_{[t] \in \CL(\smallkay)} \sum_{\substack{x \in L(t.E, \iota) \mathfrak a^{-1}\\ \snass r + x \lambda \in O_{t.E}\\ \snass
   \deg x = \frac{m}{N(\mathfrak a)}}}1
   \\
    &= C\,  |\und{\mathcal Z(m)}(\bar{\mathbb F}_p)|.
   \end{align*}
   To determine $C$,  replacing the theta function in the formula involving defining $C$ by $1$, one sees  that
     $$
   2 =\int_{H(\mathbb Q) \backslash  H(\A)} dh= C \, \sum_{t \in \CL(\smallkay)} 1 =C\, h_\smallkay
   $$
and
$$\Lambda(1, \chi) = |\Delta|^{\frac{1}2} \pi^{-1} L(1, \chi) = \frac{2 h_\smallkay}{w_{\smallkay}}.$$
So
$
C^{-1} = \frac{w_{\smallkay}}{4} \Lambda(1, \chi),
$
and
$$
|\und{\mathcal Z(m)}(\bar{\mathbb F}_p)| \,  q^{\frac{m}{N(\mathfrak a)}}=C^{-1}\, I_{\frac{m}{N(\mathfrak a)}}(\tau, \ph^E)
= \frac{w_{\smallkay}}4\, E_{\frac{m}{N(\mathfrak a)}}^*(\tau, 0;\ph^E),
$$
as claimed.
 \end{proof}

The next step is to rewrite the Eisenstein series here in terms of data that does not
involve the choice of $(E,\iota)$.  Recall that $O_E=\End(E)$ is a maximal order in the quaternion algebra
$\mathbb B$
ramified at $p$ and $\infty$, and that, \cite{KRYtiny}, pp. 376--378,  we can write $L(E,\iota) = \frak b \bar{\frak b}^{-1}\,\Cal P_0^{-1} \delta$
for a fractional ideal $\frak b$ and a certain auxillary prime ideal $\Cal P_0\mid p_0$ which is split in $\kay$.
Here $\delta \in \mathbb B$ with $\iota(\a)\delta = \delta \iota(\bar\a)$, and
$\delta^2 =\kappa$, with $\kappa = -p_0p$ if $p$ is inert in $\kay$ and $\kappa = -p_0$ if $p$ is ramified in $\kay$.
Moreover $\chi_\ell(\kappa) = 1$ for $\ell\ne p$ and $\chi_p(\kappa) = -1$. The following result is an easy consequence of this.

\begin{lem} \label{lemY4.3}  {\rm (}Howard \cite{Howard}{\rm )}
Write $L(E, \iota)\tt\Z_\ell=\delta_\ell\, O_{\smallkay, \ell}$ with $\delta_\ell^2 =\kappa_\ell\in \mathbb Z_\ell$. \hfb
(1)  If $\ell\ne p$,  $\delta_\ell$ can be chosen so that
$\kappa_\ell=1$.\hfb
(2) For $\ell =p$,   $(\Delta, \kappa_p)=-1$ and $\ord_p \kappa_p =1$ or  $0$ depending on whether $p$ is inert or ramified in $\kay$.
\end{lem}

The maximal order $O_E$ can then be described locally as follows, cf. \cite{dorman.2}, 
and Proposition~\ref{maxorders}
below.

     \begin{lem}  Let the notation be as in Lemma \ref{lemY4.3}.\hfb
 (1)  When $\ell$ is ramified in $\kay$ and $\ell \ne p$,  $\delta_\ell$,   satisfying the conditions  in Lemma \ref{lemY4.3}, can be chosen so that
     $$
      O_{E, \ell} =\{ \alpha + \beta \delta_\ell \mid\, \alpha \in \partial_\ell^{-1},  \alpha +  \beta \in O_{\smallkay, \ell}\}.
      $$
(2)  When $\ell=p$ is ramified in $\kay$,  there is a element  $\mu_p \in O_{\smallkay, p}$ with
$$\mu_p \bar{\mu}_p -\kappa_p \in p^{-1} \Delta \mathbb Z_p,$$
and such that
$$
O_{E, \ell}=\{ \alpha + \beta \delta_\ell  \mid\, \alpha, \beta \in
\varpi_p\partial_p^{-1},  \alpha + \bar{\mu}_p \beta \in
O_{\smallkay, p}\}.
$$
Here $\varpi_p$ is a uniformizer in $\kay_p$.
\end{lem}

These two lemmas allow us to identify the coherent collection
associated to $V^E$ as $\widetilde{V}^{(p)} = \{\widetilde{\Cal
V}^{(p)}_\ell\}$, for  $\widetilde{V}^{(p)}_\ell = \kay_\ell$ with
$Q_\ell(x) = \tilde\b_\ell\,N(x)$ where
$$
\tilde{\beta}_\ell=
\begin{cases}
        -1 &\hbox{if } \ell \nmid p\,\infty,
        \\
         -\kappa_p  &\hbox{if } \ell=p,
         \\
          1 &\hbox{if } \ell=\infty.
          \end{cases}
$$
When $\ell< \infty$, the local isomorphism is given by $\delta_\ell x \mapsto x$ since $L(E, \iota)\tt\Z_\ell =\delta_\ell O_{\smallkay, \ell}$.
Under this identification,  $\ph^E$ becomes $\tilde{\ph}^{(p)}=\tt_{\ell} \tilde{\ph}^{(p)}_\ell$ with $\tilde{\ph}^{(p)}_\ell(x)$ given by
\begin{equation}
\begin{cases}
        \cha(\mathfrak a_\ell^{-1}) (x)  &\hbox{if } \ell  \nmid \Delta \infty,
         \\
          \cha(\mathfrak a_\ell^{-1}) (x) \cdot \cha(-\bar r_\ell+ O_{\smallkay, \ell})( x \lambda_\ell)  &\hbox{if }  \ell  |\Delta, \ \ell\ne p,
         \\
          \cha(\mathfrak a_p^{-1}) (x) \cdot \cha(\varpi_p \partial_p^{-1})(x \lambda_p)\cdot
          \cha(-\bar r_p+ O_{\smallkay, p})\,(\mu_p x \lambda_p)  &\hbox{if }  \ell |\Delta,\ \ell=p
             \\
         e^{-2 \pi x \bar x}  &\hbox{if } \ell =\infty.
          \end{cases}
 \end{equation}
 Notice that $\tilde{\ph}^{(p)}$ does not depend explicitly\footnote{The only linkage
 is through the element $\mu_p$.}  on the choice of $(E, \iota)$. The above argument gives

  \begin{equation} \label{eqY4.5}
   E(\tau, s;\ph^E) =E(\tau, s; \tilde{\ph}^{(p)}).
   \end{equation}

To obtain our result in final form,  we rescale slightly.

\begin{lem} \label{lemY3.5}  Let $V_1 =V_2 =V$ be a vector space over $\mathbb Q_\ell$ of even dimension $n$,
with quadratic forms $Q_1(x) =a Q_2(x)$ for some $a \in \mathbb Q_\ell^\times$.
Let $\ph \in S(V)=S(V_i)$,  and let $\P_i \in I(s, \chi)$ be the associated standard sections where
$\chi=\chi_{V_1} =\chi_{V_2} =( (-1)^{\frac{n(n-1)}2} \det V_i, \, )_\ell$. Then the corresponding Whittaker functions have the following relation
   $$
   W_{m, \ell}(s_0, \Phi_1) = \frac{\gamma(V_1)}{\gamma(V_2)} |a|_\ell^{s_0} \, W_{\frac{m}a, \ell}(s_0, \Phi_2).
   $$
Here $s_0=\frac{n-2}2$ and the convention of (\ref{forshort}) is used.
 \end{lem}
\begin{proof} This follows from simple calculation using (\ref{eqWeil}).   Write $w=\kzxz {0} {1} {-1} {0}$. Then
\begin{align*}
W_{m, \ell}(s_0, \Phi_1)
 &= \int_{\mathbb Q_\ell} \omega_{V_1}(w^{-1}n(b))\ph(0)\, \psi(-m b)\, db
 \\
  &=\gamma(V_1) \int_{\mathbb Q_\ell} \int_{V} \omega_{V_1}(n(b))\ph(x) \,d_{V_1} x\, \psi(-mb)\, db
  \\
   &=\gamma(V_1) |a|_\ell^{\frac{n}2} \int_{\mathbb Q_\ell} \int_V \psi(b Q_1(x))\ph(x)\, d_{V_2} x \,\psi (-mb)\, db
   \\
    &= \gamma(V_1) |a|_\ell^{\frac{n}2-1 } \int_{\mathbb Q_\ell} \int_V \psi(b Q_2(x))\ph(x)\, d_{V_2} x\, \psi (-a^{-1} mb)\, db
    \\
      &=\frac{\gamma(V_1)}{\gamma(V_2)} |a|_\ell^{s_0}\,  W_{\frac{m}a, \ell}(s_0, \Phi_2),
      \end{align*}
      as claimed. Here $d_{V_i}$ is the Haar measure on $V_i$ self-dual with respect to quadratic forms $Q_i$.
\end{proof}

\begin{prop} \label{propY3.6}  Let $V_1=(V, Q_1)$  and $V_2=(V, Q_2)$ be positive definite
quadratic spaces over $\mathbb Q$ of even dimension $n$ with $Q_1 =a Q_2$ for
$a \in \mathbb Q_{>0}^\times$. Let $\ph\in S(V(\A_f)) =S(V_i(\A_f))$.
Let $\ph_{i, \infty} \in S(V_i(\mathbb R))$, $\ph_{i, \infty}(x) =e^{ -2 \pi Q_i(x)}$ be the Gaussian,
and let $\Phi_i \in I(s, \chi)$ be the
standard section associated to $ \ph_{i,\infty}\tt\ph \in S(V_i(\A))$, where $\chi=\chi_{V_i}$. Let $s_0= \frac{n-2}2$.
Then
$$
E_{am}(\tau, s_0; \Phi_1)\, q^{-am}=a^{-\frac{n-2}2} E_m(\tau, s_0; \Phi_2)\, q^{-m},
$$
i.e., the $(am)$-th Fourier coefficient of $E(\tau, 0; \Phi_1)$ is
the same as the $m$-th Fourier coefficient of $E(\tau, 0;\Phi_2)$,
up to a constant multiple.
\end{prop}
\begin{proof}   When $m \ne 0$, one has by Lemma \ref{lemY3.5}
\begin{align*}
E_{am}(\tau, s_0; \Phi_1)
   &= W_{am, \infty}(\tau, s_0; \Phi_{1, \infty})\,\prod_{\ell<\infty} W_{am, \ell}(s_0; \Phi_{1, \ell})
   \\
    &= W_{am, \infty}(\tau, s_0, \Phi_{1, \infty})\prod_{\ell<\infty} \frac{\gamma(V_{1,\ell})}{\gamma(V_{2, \ell})}|a|_\ell^{\frac{n-2}2}
       \prod_{\ell<\infty} W_{m, \ell}(s_0, \Phi_{2, \ell}) .
       \end{align*}
   Recall from \cite[Propostion 2.6]{KRYtiny} that $\Phi_{1, \infty}(s) =\Phi_{2, \infty}(s)$ is the normalized eigenfunction
   of weight $\frac{n}2$ and
   $$
   W_{am, \infty}(\tau, s_0, \Phi_{1, \infty})=W_{am, \infty}(\tau, s_0, \Phi_{2, \infty})= W_{m, \infty}(\tau, s_0,  \Phi_{2, \infty})\, q^{am}\, q^{-m}.
   $$
   Recall also  \cite{Rao} that
   $$
   \prod_{\ell\le \infty } \gamma(V_{1, \ell}) =\prod_{\ell\le \infty} \gamma(V_{2,\ell}) =1.
   $$
   Since $\gamma(V_{1, \infty}) =\gamma(V_{2, \infty})$,
    we have proved the proposition for $m\ne 0$. The case $m=0$ is similar and is left to the reader.
\end{proof}

For $a\in \Q^\times_{>0}$, let $\tilde{V}^{(p,a)}$ be the coherent
collection with $\tilde{V}^{(p,a)}_\ell = \kay_\ell$ and $Q_\ell(x)
= \b^{(p,a)}_\ell\,N(x)$, where
\begin{equation}
\beta^{(p, a)}_\ell =\begin{cases}
   -a  &\hbox{if } \ell \nmid p \infty,
   \\
    a  &\hbox{if } \ell = \infty,
    \\
   -\kappa_p   a &\hbox{if } \ell =  p.
    \end{cases}
    \end{equation}
  Here  $\kappa_p \in \mathbb Z_p$ with $(\Delta,  \kappa_p)_p =-1$ and $\ord_p \kappa_p =1$ or $0$,
  depending on whether $p$ is inert or ramified in $\kay$.

Let $\tilde{\ph}^{(p, a)}= \tt_{\ell} \tilde{\ph}^{(p, a)}_\ell \in
S(V^{(p, a)}(\A))$ be given by
\begin{equation}\label{phpa}
 \tilde{\ph}^{(p, a)}_\ell(x)  = \begin{cases}
 \cha(\mathfrak a_\ell^{-1})(x) &\hbox{if } \ell \nmid
 \Delta \infty,
 \\
\cha(\mathfrak a_\ell^{-1})(x)\cdot \cha( -\bar r_\ell + O_{\smallkay, \ell}) (x
\lambda_\ell) &\hbox{if } \ell\mid \Delta, \ell\ne p,
\\
\cha(\mathfrak a_p^{-1})(x)\cdot \cha(\varpi_p \partial_p^{-1})
(x\lambda_p)\cdot \cha( -\bar r_p + O_{\smallkay, p}) (\mu_p x \lambda_p)
   &\hbox{if }  \ell\mid \Delta, \ell=p,
    \\
 e^{-2 \pi  Q_\infty(x)} &\hbox{if } \ell = \infty.
   \end{cases}
   \end{equation}
In the special case where $a = N(\partial_\lambda^{-1}\mathfrak a)$,
we write $V^{(p)}=\tilde{V}^{(p,a)}$ and $\ph^{(p)}=
\tilde{\ph}^{(p, a)}$. The main result of this section is the
following.

\begin{theo} \label{counting}  Assume that $p$ is non-split in $\kay$. For any $a\in \Q^\times_{>0}$,
$$
|\und{\mathcal Z(m)}(\bar{\mathbb F}_p)|\, q^\frac{a m}{N(\mathfrak
a)}=\frac{w_{\smallkay}}{4}\,  E_{\frac{am}{N(\mathfrak a)}}^*(\tau,
0;\tilde{\ph}^{(p,a)}).
$$
In particular, for $a = N(\partial_\lambda^{-1}\mathfrak a)$, and $\mm = \frac{m}{\Dl}$,
$$
|\und{\mathcal Z(m)}(\bar{\mathbb F}_p)|\,
q^\mm=\frac{w_{\smallkay}}4 \, E_{\mm}^*(\tau, 0;\ph^{(p)}).
$$
\end{theo}

\begin{proof}
Noting that $\tilde
\ph^{(p)}=\ph^{(p, 1)}$,  we have, by Proposition~\ref{propY3.6}, (\ref{eqY4.5}), and Proposition~\ref{propY3.2},
\begin{align*}
E_{\frac{am}{N(\mathfrak a)}}^*(\tau, 0; \tilde{\ph}^{(p, a)})\,
q^{-\frac{am}{N(\mathfrak a)}}
 &= E_{\frac{m}{N(\mathfrak a)}}^*(\tau, 0; \tilde{\ph}^{(p)})\,q^{-\frac{m}{N(\mathfrak a)}}
 \\
  &= E_{\frac{m}{N(\mathfrak a)}}^*(\tau, 0; \ph^E)q^{-\frac{m}{\mathfrak a}}
  \\
  &=\frac{4}{w_{\smallkay}} \,|\und{\mathcal Z(m)}(\bF)|.
 \end{align*}
\end{proof}

Theorem \ref{counting} immediately implies

\begin{cor} If $\mathcal Z(m;\mathfrak a, \lambda, r)$ is not
empty, then it is supported at exactly one prime $p$, and
$\Diff(\mathcal V, \mm) =\{ p\}$. In  particular, if $|\Diff(\mathcal V, \mm)| >1$,
then $\mathcal Z(m; \mathfrak a, \lambda, r)$ is empty.
\end{cor}

\section{\bf Computation of the length}
\label{sectLength}

We again assume that $m>0$, and, for fix a prime $p$ that is not split in $\kay$, let $O_p=\OK\tt_Z\Z_p$, and let $\pi$ be a
uniformizer in $O_p$. For $x =(E, \iota, \bbold) \in
\ZZ(m)(\bar{\mathbb F}_p)$, we want to compute the length of the
local Henselian ring $\OO_{\ZZ(m), x}$:
\begin{equation}
\lg(x) = \hbox{length of } \OO_{\ZZ(m), x} =\hbox{length of }
\widehat{\OO}_{\ZZ(m), x}.
\end{equation}
Recall, \cite{grossqc}, that  $\widehat{\Cal O}_{\Cal C,
\mathrm{pr}(x)} = W_{O_p}(\bar\F_p)$, the ring of relative Witt
vectors.

\begin{prop}\label{length.prop}  Let the notation be as above.  Then
$$\widehat{\Cal O}_{\ZZ(m), x} = W_{O_p}(\bar\F_p)/(\pi^\nu),$$
where
$$
\nu=\nu_p(m) =
\begin{cases}
\frac12\,(\ord_p(m)+1)&\text{if $p\nmid \Delta$,}\\
\nass
\ord_p(m N(\partial \partial_\lambda^{-1})) &\text{if $p\mid N(\partial)$}.
\end{cases}
$$
In particular, $\lg(x)=\nu_p(m)$ depends only on $\ord_p m$ and $\lambda$.

\end{prop}
\begin{proof}  The point $x$ corresponds to a collection $(E,\iota, \bbold)$
over $\bar \F_p$ for a special quasi-endomorphism $\bbold \in
L(E,\iota)\frak a^{-1}$ with $N(\frak a)N(\bbold)=m$, and
$r+\bbold\l\in \End(E)$.

Let $X=E[p^\infty]$ be the $p$-divisible group  of
$E$. Then  $r+\bbold
\l$ determines an endomorphism $\xi_p =(r+\bbold \l)_p$ of  $X$.
 Gross's result, \cite[Proposition~4.3]{grossqc}, determines the deformation locus of
$(X,\iota,\xi_p)$ inside that of $(X,\iota)$ as $\Spf
\,W_{O_p}(\bar\F_p)/(\pi^\nu)$, where
$$\nu =
\frac{1}{f_p}(\ord_p(\Delta N(\xi_p^-))-1)+1.$$ Here
$f_p=[\kay(\mathfrak p): \mathbb F_p]$, and $\xi_p^- =\bbold
\lambda_p$ is the component of $\xi_p$ in $\delta\kay_p$. So
$$
N(\xi_p) = N(\bbold \lambda) =\frac{m}{N(\mathfrak a)} N(\mathfrak
a \partial_{\lambda}^{-1}) = \mm,
$$
and
$$
\nu=
\begin{cases}
\frac12\,(\ord_p(m)+1)&\text{if $p\nmid \Delta$,}\\
\nass
\ord_p(m \Delta \Dl^{-1}) &\text{if $p\mid \Delta$},
\end{cases}
$$
as claimed.
\end{proof}

Combining this proposition with Theorem  \ref{counting},  we obtain the following intermediate result.

\begin{theo}  \label{theodegree} For a positive integer $m>0$ with
$\Diff(\mathcal V,\mm)=\{p\}$, and any positive rational number $a>0$, then
$$
\widehat{\deg}\, \mathcal Z(m)\,  q^{\frac{am}{N(\mathfrak a)}}=
\frac{1}4 c_p(m) \log p  \cdot E_{\frac{am}{N(\mathfrak
a)}}^*(\tau, 0, \ph^{(p, a)}) .
$$
Here
$$
c_p(m) =  \begin{cases}
\ord_p(m)+1 &\text{if $p\nmid \Delta$,}\\
\nass
\ord_p(m \Delta \Dl^{-1}) &\text{if $p\mid \Delta$}.
\end{cases}
$$
\end{theo}

\section{\bf Whittaker functions and their derivatives}
\label{sectWhittaker}

  \begin{prop} \label{propY5.1}  Suppose $p \in \Diff(\mathcal V, \frac{ m}{\Dl})$. Then
  $$
  W_{\mm, p}^*( 0, \ph^{(p)}_p)\, E_{\mm}^{*, \prime}(\tau, 0, \ph)
   =W_{\mm, p}^{*, \prime}( 0, \ph_p)\, E_{\mm}^*(\tau, 0,
   \ph^{(p)}).
   $$
   Here $\ph$ is defined in (\ref{mainph}) and $\ph^{(p)}$ is defined in (\ref{phpa}).
\end{prop}
\begin{proof} Notice first that $\mathcal V_\ell = V^{(p)}_\ell$ as quadratic spaces and $\ph_\ell
=\ph^{(p)}_\ell$ for all $\ell \ne p$. Since $p \in \Diff(\mathcal V, \mm)$,
$W_{\mm, p}^*(0, \ph_p)=0$, and so
\begin{align*}
& W_{\mm, p}^*(0, \ph^{(p)}_p) \,E_{\mm}^{*, \prime}(\tau, 0,
\ph)
\\
 &=W_{\mm, p}^*(0, \ph^{(p)}_p) \,W_{\mm, p}^{*, \prime}(0,
 \ph_p) \prod_{\ell \ne p}W_{\mm, \ell}^{*}(0,
 \ph_\ell)
 \\
  &=W_{\mm, p}^{*, \prime}(0,
 \ph_p) \prod_\ell W_{\mm, \ell}^{*}(0,
 \ph^{(p)}_\ell)
 \\
 &=W_{\mm, p}^{*, \prime}(0, \ph_p)\, E_{\mm}^*(\tau, 0,
   \ph^{(p)}).
   \end{align*}
\end{proof}

To compute the central derivative of the Whittaker functions, we
need the following few lemmas,  whose proofs can be found in
\cite{HowardYang}. Special cases can be found in \cite{KRYtiny},
\cite{YaValue}, and \cite{KYeis}.

 Let $K$ be a $p$-adic local field, and let $L$ be a
quadratic extension (including $K \oplus K$) with associated
quadratic character $\chi$. Let $\psi_K$ be a fixed unramified
character of $K$.  For $t \in O_K$, $t\ne0$, let $V_t=(L, t x \bar
x)$ be the corresponding binary quadratic space, and, for $\mu \in
\partial_{L/K}^{-1}$, let $\phi^\mu= \cha(\mu +O_L)\in S(V_t)$. Let $\pi$ be a uniformizer of $K$ and $q=|O_K/\pi|$. Let
$$
W_{m}^*(s, \phi^\mu) = |d_{L/K}|^{\frac{s+1}2} L(s+1, \chi)\,W_{m}(s, \phi^\mu)
$$
be the normalized Whittaker function.

The following is well known, see for example \cite{KRYtiny},
\cite{YaValue}.

\begin{lem}  \label{lemY5.2} Suppose that $L/K$ is unramified and that $t \in O_K^\times$.
If $m\notin O_K$, then $W_{m}^*(s, \phi^0) =0$.  If $m \in  O_K$, then
$$
\gamma(V_t)^{-1}\,W_m^*(s, \phi^0)
 = \sum_{r=0}^{\ord_K m} (\chi(\pi) q^{-s})^r.
 $$
 In particular, $W_{m}^*(0, \phi^0)=0$ if and only if $
 \chi(m)=-1$ and $\ord_K m$ is even, and, in this case,
 $$
\gamma(V_t)^{-1}\,W_m^{*, \prime}(0, \phi^0) = \frac{1}2(1+
\ord_K m) \log q.
$$
\end{lem}

\begin{lem}  \label{lemY5.3} {\rm (}\cite[Lemma  4.6.3]{HowardYang}{\rm )} Suppose that $L/K$ is ramified, and let $N=\ord_K
m$, $c= \ord_K t$, and $X=q^{-s}$.

 (1)\  If $m
  \notin O_K$, then $W_m^*(s, \Phi_t^0)=0$.  If $m
  \in O_K$, then
  $$
\gamma(V_t)^{-1}W_m^*(s,  \phi_t^0)
 =\begin{cases}
   |t| (1-X) \sum_{n=0}^{N} (qX)^n &\hbox{if }
   N < c,
   \\
    |t|  (1-X) \sum_{n=0}^{c-1} (qX)^n +  (X^{c} + \chi(t m)
     X^{f+ N })  &\hbox{if }  N \ge c.
     \end{cases}
$$
   (2)\  Suppose that $m \in O_K$, $m\ne 0$, then
$$
\gamma(V_t)^{-1}\,W_m^*(0, \phi_t^0) = \begin{cases}
   0 &\hbox{if } N< c,
   \\
    1+\chi(m t) &\hbox{if } N \ge c.
    \end{cases}
$$
  In particular,  $W_{m}(0,
   \phi_t^0) =0$ if and only if $N < c$ or
   $\chi(tm ) =-1$. In this case
   $$
\gamma(V_t)^{-1}\,W_m^{*, \prime}(0, \phi_t^0)
 =\log q\cdot
 \begin{cases}
 q^{-c} \sum_{n=0}^{N} q^n &\hbox{if } N < c,
   \\
 \frac{1-q^{-c}}{q(1-q^{-1})} +  (N +f  -c) &\hbox{if } N \ge
  c.
\end{cases}
$$

\end{lem}

\begin{lem} \label{lemY5.4} {\rm(}\cite[Lemma 4.6.4]{HowardYang}{\rm)}
Let the notation be as in  Lemma \ref{lemY5.3} and  assume  that $p
\ne 2$ and that $\mu \notin O_K$. If $m \notin t \mu \bar \mu +O_K$, then $W_m^*(s,
\phi_t^\mu)=0$. If  $m \in t \mu \bar \mu +O_K$, write $c(m,
\mu)=\ord_\pi (m -t \mu \bar \mu)$. Then
$$
\gamma(V_t)^{-1}\,W_m^*(s,  \phi_t^\mu)=\begin{cases}
 |t| (1-X)
\sum_{0 \le n \le  c(m, \mu)}(qX)^n, &\hbox{if } c(m,
\mu) < c, \\
|t|  (1-X) \sum_{0 \le n  <c} (qX)^n +  X^c , &\hbox{if } c(m, \mu)
\ge c. \end{cases}
$$

In particular, $W_m^*(0, \phi_t^\mu)=0$ if and only if $c(m,
\mu) <c$. In this case,
$$
\gamma(V_t)^{-1}\,W_m'(0, \phi_t^\mu)=|t| \log q \sum_{0 \le
n \le c(m, \mu)}q^n.
$$
\end{lem}

   \begin{prop} \label{Whittaker} Let  $ m >0$ be a rational number, and assume that $p \in \Diff(\mathcal V, \mm)$.\hfb
(1) \ Suppose that $p$ is inert in $\kay$. If $m \notin \mathbb Z_p$, then $W_{m,p}^*(s, \ph^{(p)}) = 0$. If $m \in \mathbb Z_p$, then
   $$
   \frac{W_{\mm,p}^{*, \prime}(0, \ph_p)}{W_{\mm, p}^*(0, \ph^{(p)}_p)}  =-\frac{1}2 (1 + \ord_p m) \log p.
   $$
(2) \  If $p \ne 2 $ is ramified in $\kay$ and $r_p \notin O_{\smallkay, p}$,
$$W_{m,p}^*(s, \ph^{(p)})
  =0, \qquad\text{and}\qquad W_{\mm, p}^*(s, \ph_p) =0.$$
(3) \  Suppose that $p \ne 2 $ is ramified in $\kay$ and that $r_p \in O_{\smallkay, p}$. If $\mm \notin \mathbb Z_p$, then
    $W_{\mm ,p}^*(s, \ph^{(p)}) = 0$.  If $\mm \in \mathbb Z_p$,  then
    $$
   \frac{W_{\mm ,p}^{*, \prime}(0, \ph_p)}{W_{\mm , p}^*(0, \ph^{(p)}_p)}
   =-\frac{1}2 \ord_p( m N(\partial \partial_\lambda^{-1}))  \log p.
   $$
 (4) \ In all cases,
 $$
 W_{\mm ,p}^{*, \prime}(0, \ph_p) = -\frac{1}2\, c_p(m)\, \log p \cdot  W_{\mm , p}^*(0, \ph^{(p)}_p),
 $$
 where $c_p(m)$ is the number given in Theorem \ref{theodegree}.
   \end{prop}
   \begin{proof} First assume that $ p \nmid \Delta$ is unramified in $\kay$.
   Then we can choose $\alpha \in \mathfrak a_p$ with $N_{\smallkay_p/\mathbb Q_p} \alpha = N(\mathfrak a \partial_\lambda^{-1})$ in $\mathbb Q_p^\times$.
   So $x \mapsto x\alpha$ gives an  isomorphism of quadratic space
   $ V_p^{(p)} =(\kay_p, -N(\mathfrak a \partial_\lambda^{-1}) x \bar x)$ to
   $\tilde{\mathcal V}_p=(\kay_p, - x \bar x)$, under which,
   $\ph_p =\cha(\mathfrak a_p^{-1})$ becomes   $\tilde \ph_p=\cha(O_{\smallkay, p})$.
   Then, by Lemma \ref{lemY5.2},  noting that $\ord_p \partial_\lambda =0$,
   \begin{align*}
   \gamma( V_p^{(p)})^{-1}\,W_{\mm, p}^*(s, \ph_p)&= \gamma(\tilde{\mathcal V}_p)^{-1}\, W_{\mm, p}^*(s,
   \tilde{\ph}_p)
   \\
    &= \sum_{0 \le r \le \ord_p m} (\chi_p(p) p^{-s})^r.
    \end{align*}
    For the same reason, $x \mapsto x\alpha$ gives  isomorphism of quadratic space
    $ V^{(p)}_p =(\kay_p, -\kappa_p N(\mathfrak a \partial^{-1}) x \bar x)$ to
    $\tilde{\mathcal V}_p=(\kay_p, - \kappa_p x \bar x)$, under which,
    $\ph^{(p)}_p =\cha(\mathfrak a_p^{-1})$ becomes   $\tilde{\ph}^{(p)}_p=\cha(O_{\smallkay, p})$. So Lemma \ref{lemY5.2} gives
    $$
    \gamma(V^{(p)}_p)^{-1}\,W_{\mm, p}^*(0, \ph^{(p)}_p)
    = \gamma(\tilde{\Cal V}_p)^{-1}\,W_{\mm, p}^*(0, \tilde{\ph}^{(p)}_p) =\cha(\mathbb Z_p)(m).
    $$
    Here we have used the fact that $p \in \Diff(\mathcal V, \mm)$, which implies that
    $\chi_p(\mm) =-1$.
    So $W_{\mm, p}^*(0, \ph^{(p)}_p)\ne 0$ if and only if $m\in \mathbb Z_p$, and, in this case, we have
     $$
   \frac{W_{\mm,p}^{*, \prime}(0, \ph_p)}{W_{\mm, p}^*(0, \ph^{(p)}_p)}
    =-\frac{1}2 (1 + \ord_p m) \log p,
   $$
where the negative sign comes from the fact
$$
\gamma(V^{(p)}_p) =-\gamma(\mathcal V_p).
$$

   Next, we assume that $p\ne 2$ is ramified. Then  $\partial_p =\varpi \,O_{\smallkay, p}$, and so
   \begin{align*}
   \ph^{(p)}_p(x) &= \cha(\mathfrak a_p^{-1})(x)\cdot \cha(\varpi_p \partial^{-1})(x\lambda_p)\cdot
      \cha(-\bar{r}_p + O_{\smallkay, p}) (\mu_p x \lambda_p)
      \\
       &= \begin{cases} 0 &\hbox{if }  \ord_p r_p <0,
       \\
         \cha(O_{\smallkay, p})(x\lambda_p) &\hbox{if } \ord_p r_p =0 .
         \end{cases}
         \end{align*}
         On the other hand,
        \begin{align*}
  \ph_p&=\cha(\mathfrak a_p^{-1})(x)\cdot
      \cha(-\bar r_p + O_{\smallkay, p}) ( x \lambda_p)
      \\
      &=\begin{cases}
\cha(\mathfrak a_p^{-1})(x)\cdot
      \cha(-\bar r_p + O_{\smallkay, p}) ( x \lambda_p) &\hbox{ if }  \ord_p r_p <0,
        \\
         \cha(O_{\smallkay, p})(x \lambda_p)  &\hbox{ if } \ord_p r_p =0 .
         \end{cases}
         \end{align*}

To prove  (2), we first assume that $r_p \notin O_{\smallkay, p}$. This implies
 that $\lambda_p$ generates $\partial_p^{-1} \mathfrak a_p$,  i.e., that $p |\Dl$. In
 this case,
 $x \mapsto x \lambda_p$ gives an isomorphism  from $\mathcal V_p$ to $(\kay_p, - \frac{N(\mathfrak a \partial_\lambda^{-1})}{\lambda_p \bar{\lambda}_p} x \bar x)$, under which $\ph_p$ becomes
      $\tilde{\ph}_p=\cha(-\bar r_p + O_{\smallkay, p})$. Set $t =- \frac{N(\mathfrak a \partial_\lambda^{-1})}{\lambda_p \bar{\lambda}_p}$, and note that $\ord_p t
      =0$.
     Lemma \ref{lemY5.4}
      shows that
      $W_{m, p}^*(s, \ph_p) =0$ unless
      $$
      \mm - t (r_p \bar r_p) \in  \mathbb Z_p,
 $$
   which  is impossible since $p \in \Diff(\mathcal V, \mm)$.

    To prove (3), we next assume that $r_p \in O_{\smallkay, p}$.  In this
     case,  the map $x \mapsto x \lambda_p$ gives an isomorphism
     between $\mathcal V_p$ and $(\kay, - \frac{N(\mathfrak
     a \partial_\lambda^{-1})}{\lambda_p \bar{\lambda}_p} x \bar x)$, under which
     $\ph_p$ becomes $\cha(O_{\smallkay, p})$.  Set $t =\frac{N(\mathfrak
     a \partial_\lambda^{-1})}{\lambda_p \bar{\lambda}_p} \in \mathbb Z_p^\times$. Then
     Lemma \ref{lemY5.3}  shows that
 $W_{\mm, p}^*(s, \ph_p)=0$ unless $\mm \in \mathbb Z_p$. In that case, the same lemma shows that
$$
  \gamma(\mathcal V_p)^{-1}\,W_{\mm,p}^{*, \prime}(0, \ph_p)
     = \ord_p (\mm \Delta ) \log p,
     $$
     and
     $$
\gamma(V^{(p)}_p)^{-1}\,W_{\mm, p}^*(0, \ph^{(p)}_p)
  =
    2.
   $$
   Since
   $
   \gamma(\mathcal V_p) =-\gamma(V^{(p)}_p),
   $ the proposition is proved.
 \end{proof}

{\bf Proof of Theorem \ref{maintheo1} for $m>0$.}   
We may assume that $m$ is an integer and
$\Diff(\mathcal V, \mm)=\{p\}$,  since otherwise
both sides are zero. Furthermore, we assume $r_p \in O_{\smallkay,
p}$ and $ \mm \in \mathbb Z_p$, since otherwise
both sides are zero by Proposition \ref{Whittaker}.  Under these
conditions,  $W_{\mm, p}^*(0,
\ph^{(p)}_p) \ne 0$ by  Proposition \ref{Whittaker}.  Now Theorem
\ref{theodegree} with $a=N(\mathfrak a \partial_\lambda^{-1})$
gives
$$
\widehat \deg\, \mathcal Z(m)\, q^{\mm}
   =\frac{1}4\,  c_p(m)\, E_{\mm}^*(\tau, 0,  \ph^{(p)}).
   $$
   On the other hand,  Propositions \ref{propY5.1} and \ref{Whittaker} give
\begin{align*}
E_{\mm}^{*, \prime}(\tau, 0, \ph)
   &=
  \frac{W_{\mm, p}^{*, \prime}( 0,
\ph_p)  }{W_{\mm, p}^*(0, \ph^{(p)}_p) }\, E_{m}^*(\tau, 0,
\ph^{(p)})
\\
 &=-\frac{1}2\, c_p(m)\, E_{m}^{*, \prime}(\tau, 0, \ph)
 \\
  &= -2 \, \widehat \deg\, \mathcal Z(m)\, q^{\mm}.
  \end{align*}
as claimed.

\begin{rem} Instead of $\mathcal V$ and $\ph$, we can use $\mathcal V^{(a)}$ and $\ph^{(a)}$ for any rational number $a>0$. The only problem is  that
the identity
$$
\frac{W_{\frac{am}{N(\mathfrak a)}, p}^{*, \prime}(0,  \ph^{(a)}_p)}{W_{\frac{am}{N(\mathfrak a)}, p}^*(0, \ph^{(p, a)}_p)}
  =\frac{1}2\, c_p(m)\, \log p
$$
 is not true in general. One can fix this by using a modified Eisenstein series, as was done in  \cite{KRYbook}.
 Let $S$ be the set of primes dividing either the numerator or denominator of $\frac{a \Dl}{N(\mathfrak a)}$, then there are
 holomorphic functions  $c_p(s)$ with $c_p(0) =0$ and coherent Eisenstein series $E^*(\tau, s, \ph^{(p), \prime})$ such that the modified Eisenstein series
 $$
 \tilde E(\tau, s,  \ph^{(a)}) = E^*(\tau, s, \ph^{(a)}) + \sum_{p \in S} c_p(s) E^*(\tau, s, \ph^{(p), \prime})
 $$
 satisfies
 $$
 \widehat{\deg} \mathcal Z(m; \mathfrak a, \lambda, r)\, q^{\frac{am}{N(\mathfrak a)}} =\tilde E_{\frac{am}{N(\mathfrak a)}}'(\tau, 0, \ph^{(a)}).
 $$
\end{rem}

\section{\bf The case $m<0$}
\label{sectY6}

\newcommand{\VV}{V^{(\infty)}}

In this section, we prove Theorem~\ref{maintheo1} in the case $m<0$.
This case is simpler than the case $m>0$ since the `length' is artificially
defined to be independent of the `points' $(E, \iota, \bbold)$.
The calculation of $|Z_\mathbb C(m)|$ can be dealt exactly the same way as
in Section  \ref{sectCounting}, so we only give sketch of the
proof.

Let $\VV =(\kay_\A, - N(\mathfrak a \partial_\lambda^{-1}) x \bar x)$ be the coherent
quadratic space over $\A$. Notice that $\VV$ differ from $\mathcal
V$ at exact the prime $\infty$.

We define $\ph^{(\infty)}=\prod \ph^{(\infty)}_\ell \in S(\VV)$ as
follows.
  \begin{equation}
  \ph^{(\infty)}_\ell(x) = \begin{cases}
     \cha(\mathfrak a_\ell^{-1})(x) &\hbox{if } \ell \nmid \Delta \infty,
     \\ \nass
      \cha(\mathfrak a_\ell^{-1})(x)\cdot \cha(-\bar r_\ell + O_{\smallkay, \ell})(x \lambda_\ell) &\hbox{if } \ell |\Delta ,
      \\ \nass
      e^{-2 \pi  N (\mathfrak a \partial_\lambda^{-1}) x \bar x} &\hbox{if }  \ell =\infty.
      \end{cases}
      \end{equation}
     Then we have
     \begin{prop}
      $$
      |Z_\mathbb C(m)| \, q^{\mm}= \frac{w_{\smallkay}}{4}\, E_{\mm}^*(\tau, 0,
      \ph^{(\infty)}).
      $$
     \end{prop}
\begin{proof} (sketch) Fix $(E, \iota) = ((\kay \otimes_{\mathbb Q} \mathbb R) /\OK, \iota) \in \mathcal C(\mathbb C)$ with
the fixed embedding of $\kay$ into $\mathbb C$ giving the complex
structure on $\kay \otimes_{\mathbb Q} \mathbb R$. Let
$j_0(r\otimes x) = \bar r \otimes x$. Then $j_0 \in L(E^{\Top},
\iota)$, and $j_0^2=1$. Moreover, $L(E^{\Top}, \iota) =j_0
\OK$. Since $\CL(\kay)$ acts on $\mathcal C(\mathbb C)$
simply transitively,  the rest is exact the same argument as in
Section \ref{sectCounting}.
\end{proof}

 \begin{prop}  (1) \  For $m <0$, one has
 $$
W_{\mm, \infty}^*(\tau, 0, \ph^{(\infty)}_\infty)\, E_{\mm}^{*, \prime}
(\tau, 0,  \ph) =W_{\mm, p}^{*, \prime}(\tau, 0, \ph_\infty)\,
      E_{\mm}^*(\tau, 0, \ph^{(\infty)}).
 $$
(2) \  For $m<0$,  $W_{\mm, \infty}^*(\tau, 0,
\ph^{(\infty)}_\infty)\ne 0$, and
$$
\frac{W_{\mm, p}^{*, \prime}(\tau, 0,
\ph_\infty)}{W_{\mm, \infty}^*(\tau, 0,
\ph^{(\infty)}_\infty)} =-\frac{1}2 \beta_1(4 \pi |\mm|
v).
$$
 \end{prop}
\begin{proof} (1) is exactly the same as Proposition \ref{propY5.1}.

(2)  Notice that $\VV_{\infty} = -\mathcal V_\infty$. Thus, as in
Lemma \ref{lemY3.5} and \cite[Proposition 2.6]{KRYtiny}
$$
\gamma(\VV_{\infty})^{-1}\,W_{\frac{m}{\Delta(\lambda)},
\infty}^*(\tau, 0, \ph^{(\infty)}_\infty)\,
q^{-\frac{m}{\Delta(\lambda)}}
 = \gamma(\mathcal V_\infty)^{-1}\,W_{-\frac{m}{\Delta(\lambda)}, \infty}^*(\tau, 0, \ph_\infty)\, q^\frac{m}{\Delta(\lambda)}
 =2.
$$
  By the
same proposition, one has
$$
\gamma(\mathcal V_\infty)^{-1}\,W_{\mm, \infty}^{*, \prime}(\tau, 0, \ph_\infty)\,
q^{-\mm} =\beta_1(4 \pi |\mm| v)
$$
This proves (2) since
$
\gamma(\mathcal
 V_\infty) = -\gamma(\VV_{\infty}).
$

{\bf Proof of Theorem \ref{maintheo1} for $m<0$}. Now the proof for
the case $m<0$ is the same as that of $m>0$. We leave the details
to the reader.
\end{proof}

\medskip
\medskip
\medskip

\centerline{\Large\bf Part II}

\section{\bf Maximal orders and optimal embeddings}\label{section.maxorders}

In this section, we summarize the results we need concerning maximal orders in
a quaternion algebra $B$ over $\Q$
with an optimal embedding of $\OK$.
Of course, some of this material is rather classical and well known, \cite{gross.zagier.singular}, \cite{dorman.1},
\cite{dorman.2}, but we need certain detailed information for which we found no good reference.
In this section, we do not assume that $B$ is indefinite.

Let $\Delta$ be the discriminant of $\OK$, and let $\chi$ be the Dirichlet character associated to $\kay$.
Let $D=D(B)$ be the product of the primes that ramify in $B$ and write $D(B) = D_0D_1$, where $D_1$
is the product of the primes which ramify in both $\kay$ and $B$. The primes dividing $D_0$ are all inert in $\kay$.
Let $\d$ be the
different of $\kay/\Q$ and let $\d_1\mid \d$ be the factor with $N(\d_1)=D_1$.
Write $\d=\d_1\d_2$, and note that,
if $2$ is ramified in both $B$ and $\kay$, then $\d_1$ and $\d_2$ are not
relatively prime.
Let $\Delta_2=N(\d_2)$.

Fix an embedding $i:\kay\rightarrow B$ and write
\begin{equation}\label{Bbasis}
B = \kay+\kay\,\delta, \qquad \text{with}\
\delta^2=\kappa \in \Q^\times,
\end{equation}
where $\delta \,\a = \bar\a\,\delta$ for $\a\in \kay$.  The element $\delta$ is unique up to scaling by an element of $\kay^\times$.
Note that  $(\kappa,\Delta)_p = \inv_p(B)$.
We write $[\a,\b] = \a+\b\delta$ and  obtain an embedding
\begin{equation}\label{Bembed}
B\hookrightarrow M_2(\kay), \qquad [\a,\b]
\mapsto \begin{pmatrix} \a&\b\\\kappa\bar\b&\bar\a
\end{pmatrix}.
\end{equation}
Note that we have an isomorphism
\begin{equation}\label{equiiso}
B\isoarrow \kay^2, \qquad [\a,\b] \mapsto \begin{pmatrix} \b\\ \bar\a\end{pmatrix}
= \begin{pmatrix} \a&\b\\\kappa\bar\b&\bar\a
\end{pmatrix}\,\begin{pmatrix} 0\\ 1\end{pmatrix},
\end{equation}
equivariant for the left action of $B$ and such that
$$[\a,\b]\,[\a_0,0] \mapsto \bar\a_0 \begin{pmatrix} \b\\ \bar\a\end{pmatrix},$$
i.e., conjugate linear for the right action of $\kay$.  Note that this action is the restriction to $B\tt\kay$
of the action of $B\tt B$ on $B$ defined by
$b_1\tt b_2: x \mapsto b_1\,x\, b_2^{\iota}.$

Let $\Max_B(\OK)$ be the set of maximal orders $O_B$ in $B$ with $i^{-1}(O_B)=\OK$.
Such orders can be described as follows.  Consider pairs $(\frak a, \l)$ where $\frak a$ is
a fractional ideal with
\begin{equation}\label{akapparelation}
N(\frak a) =\kappa\frac{|\Delta|}{D},
\end{equation}
and  $\l$ is a generator of the cyclic $\OK$-module $\d_2^{-1}\frak a/\frak a$ such that
\begin{equation}
N(\l)\equiv \kappa \text{ in }  \Delta_2^{-1}N(\frak a)/N(\frak a).
\end{equation}
For a given pair $(\frak a,\l)$, let
\begin{equation}\label{classicalmax}
O_{\frak a,\l,B} = \{\ [\a,\b]\mid  \a\in \d_2^{-1}, \ \b\in \frak a^{-1}, \
\a+\l\b \equiv 0\  \text{in}\  \d_2^{-1}/\OK\ \}.
\end{equation}

The following result is a slight generalization of the description of maximal orders given
\cite{dorman.1} and \cite{dorman.2}.

\begin{prop}\label{maxorders}
(i)  $O_{\frak a,\l,B}$ is a maximal order in $\Max_B(\OK)$.\hfb
(ii) Every maximal order in $\Max_B(\OK)$ has the form
$O_B= O_{\frak a,\l,B}$ for some $\frak a$ and $\l$. \hfb
(iii) The finite id\`eles\footnote{imbedded in $(B\tt\A_{\Q,f})^\times$ via $i$} $\A^\times_{\smallkay,f}$ act on the set $\Max_B(\OK)$
by conjugation,
$$O_B \mapsto \bbbold\, O_B \bbbold^{-1}=\bbbold\, \widehat O_B \bbbold^{-1}\cap B,$$
where $\bbbold\in \A^\times_{\smallkay,f}$ and $\widehat O_B = O_B\tt_Z\widehat \Z$.
Explicitly,
$$\bbbold\, O_{\frak a,\l,B} \,\bbbold^{-1} = O_{\frak a_\bbbold,\l_\bbbold,B},$$
where
$$\frak a_\bbbold = \frak b \bar{\frak b}^{-1}\frak a, \qquad\text{and}\qquad
\l_{\bbbold} = \bbbold \bar\bbbold^{-1}\l,$$
and $\frak b$ is the fractional ideal determined by $\bbbold$.
This action is transitive. \hfb
(iv)  Suppose that $\delta$ is replaced by $\delta'=\b_0\delta$, for $\b_0\in \kay^\times$. Then,
$$O_B =O_{\frak a,\l,B} = O_{\frak a',\l',B}$$
where $\frak a' = \b_0\frak a$, and $\l'=\b_0\l$.
\end{prop}

Let
\begin{equation}
\L(\frak a, \kappa, \Delta_2) = \{ \ \l\in \d_2^{-1}\frak a/\frak a \mid  \text{$\l$ a generator},
N(\l)\equiv \kappa\ \text{in}\  \Delta_2^{-1}N(\frak a)/N(\frak a)\ \}
\end{equation}
be the set of possible choices of $\l$. For a place $w\mid \d_2$, let $\l_w$ be the image of $\l$
in $(\d_2^{-1}\frak a/\frak a)_w$.
\begin{lem} (i)
For each $w\mid \d$,  with $w\nmid \d_1$,
there are two choices
$\l_w$ and $\varpi_w\bar\varpi_w^{-1}\l_w$ of the
local component
$\l_w\in ( \d_2^{-1}\frak a/\frak a)_w$ of $\l$.
Here $\varpi_w$ is a local uniformizer at $w$. \hfb
(ii) If $w\mid \d_2$ and $w\mid \d_1$ {\rm (}i.e., $w\mid 2$ is ramified and $2\mid D_1${\rm )},
then there is a unique choice of $\l_w$. \hfb
(iii) In particular,
$$|\L(\frak a,\kappa,\Delta_2)| = 2^{o(\Delta)-o(D_1)}$$
where $o(\Delta)$ {\rm(}resp. $o(D_1)${\rm )} is the number of prime factors of $\Delta$ {\rm(}resp.  $D_1${\rm )}.
\end{lem}
Note that for a odd place $w$, the possible local components are just $\pm\l_w$. For $w\mid 2$,
an elementary calculation gives the result.

Given a maximal order $O_B\in \Max_B(\OK)$, we
will be interested in $\OK$-lattices $\L$ in $\kay^2$ that are stable under the action of $O_B \subset M_2(\kay)$.
Note that, for such a lattice $\L$, we have
$O_B = B\cap O_\L,$
where $O_\L\subset M_2(\kay)$ is the maximal order stabilizing $\L$.
Under the isomorphism (\ref{equiiso}), such lattices correspond to left $O_B$-ideals
that are stable under the right action of $\OK$.  The finite id\'eles $\A^\times_{\smallkay,f}$
act on the set of such ideals by right homotheties,
$$I\mapsto I\cdot \bbbold = \widehat I \cdot \bbbold \cap B.$$
Since the right order of such an ideal is again an element of
$\Max_B(\OK)$,  part (iii) of Proposition~\ref{maxorders} implies that every $\A^\times_{\kay,f}$-orbit
contains a two-sided $O_B$ ideal. For a subset $S$ of the set of primes dividing $D_0$, let
$$I_S = \prod_{p\in S} \Cal P,$$
where $\Cal P$ is the two sided $O_B$-ideal with $\Cal P^2 = pO_B$.
Then the $I_S$'s are a set of representatives for the $\A^\times_{\smallkay,f}$-homothety classes.
In particular, there are $2^{o(D_0)}$ such classes.
Let $\L_S$ be the image of $I_S$ under the isomorphism (\ref{equiiso}).
Then $\L_{O_B}:=\L_{\emptyset}$ is the image of $O_B$.

Any $\OK$-lattice in $\kay^2$ has the form $\gbold\,\L_0$, where $\L_0=\OK^2$ and
$\gbold\in \GL_2(\A_{\smallkay,f})$.
Choose a finite id\`ele $\abold$ with
$\ord_w(\abold) = \ord_w(\frak a)$ for all $w$, and an ad\`ele $\lbold$
with components $\lbold_w$ having image $\l_w$ in $(\d_2^{-1}\frak a/\frak a)_w$
for $w\mid \d_2$ and $\lbold_w=0$ otherwise.  We will sometimes write $O_{\abold,\lbold, B}$ in place of $O_{\frak a, \l,B}$.
Then a short calculation shows that, for $O_B=O_{\abold,\lbold, B}$,
\begin{equation}\label{lattice.def}
\L_{O_B}=\gbold\, \L_0,\qquad\text{where}\quad \gbold =\abold^{-1} \begin{pmatrix}1&{}\\
\lbold&\abold\end{pmatrix}.
\end{equation}

 \section{\bf Morphisms $j_\L:\Cal C \lra \M$}\label{morphismsection}

In this section, we suppose that an indefinite quaternion algebra $B$ with a fixed maximal order $O_B$
and an imaginary quadratic field $\kay=\Q(\sqrt{\Delta})$,
with an embedding $i:\OK\rightarrow O_B$ are given.
If $\L\subset B$ is a left $O_B$-ideal 
which is stable for the right action of $\OK$, we view $\L$ as an $O_B\tt_{\Z}\OK$-module
by the rule\footnote{The $\bar{\a}$ occurs here due to the conventions of section~\ref{section.maxorders} above.}
\begin{equation}\label{righttwist}
(b\tt \a): x\mapsto b x\bar\a.
\end{equation}
Then there is a functor, given by Serre's construction, \cite{serre.CF}, \cite{xue}, 
$E \ \mapsto \ A=\L\tt_{\OK} E$,
from elliptic curves $(E,\iota)$ with $\OK$ action over an $\OK$-scheme $S$  to $O_B$-modules $(A,\iota_B)$
over $S$.

Fix an element $\delta$ as in section~\ref{section.maxorders} and hence an identification
$B\simeq \kay^2$ as in (\ref{equiiso}).  Let $\L_0\subset B$  be the free $\OK$-module corresponding to $\OK^2\subset \kay^2$.
Then there is a
natural quasi-isogeny
\begin{equation}\label{quasiiso}
A_0 = \L_0\tt_{\OK}E \lra \L\tt_{\OK}E = A,
\end{equation}
and $A_0 = E\times_S E$.
\begin{lem} The $O_B$-module $A = \L\tt_{\OK}E$ satisfies the Drinfeld
special condition or Kottwitz condition (3.1.2)
in \cite{KRYbook}.
\end{lem}
\begin{proof} The action of $\OK$ on $\Lie(E/S)$ satisfies
$$\text{\rm char}(\iota(\a)|\Lie(E))(T) = T-i(\a) \in \Cal O_S[T],$$
for $\a\in \OK\overset{}{\rightarrow} \Cal O_S$.
Thus, for $\xi\in M_2(\OK)$,
$$\text{\rm char}(\iota(\xi)|\Lie(E\times E))(T) = T^2-\tr(i(\xi))T+\det(i(\xi))\in \Cal O_S[T],$$
and, hence, for $b\in O_B$,
$$\text{\rm char}(\iota(b)|\Lie(A))(T) = T^2-\tr(b)T+\nu(b)\in \Cal O_S[T],$$
as required, since this formula is unchanged under the quasi-isogeny from $A_0$ to $A$.
\end{proof}
Thus there is a morphism
$$j_\L: \Cal C \lra \M, \qquad (E,\iota) \mapsto (A,\iota_B),$$
of moduli stacks.

Let $O_\L = \End_{\OK}(\L)$ be the stabilizer of $\L\subset \kay^2$ in $M_2(\kay)$.
\begin{lem}  For $A = \L\tt_{\OK} E$,
$\End_{\OK}(A/S) = O_\L.$
\end{lem}
\begin{proof}
By functorial
properties of the Serre construction,  there is a natural map
$$O_\L \lra \End_{\OK}(A/S), \qquad \phi \mapsto \phi\tt 1_E,$$
and
$$\End_{\OK}(A/S) = \End_{\OK}((\L\tt_{\OK}E)/S)\simeq O_\L\tt _{\OK}\End_{\OK}(E/S) = O_\L.$$
since $\End_{\OK}(E/S) = \OK$.
\end{proof}

Note that for any $\abold\in \A^\times_{\smallkay,f}$ there is a functor $F_\abold:\Cal C \lra \Cal C$
given by $E\mapsto (\abold)\tt_{\OK} E$. Then, keeping in mind the conjugation in
(\ref{righttwist}),  we have the relation
$j_{\L\,\abold} = j_\L\circ F_{\bar\abold}.$
Thus we need only consider the morphisms $j_\L$ for $\L$ in the set of representatives for the
$\A^\times_{\smallkay,f}$-homothety classes of lattices described in the previous section.

Next suppose that $I$ is a two-sided $O_B$-ideal.  By a (slightly noncommutative) analogue
of the Serre construction, there is a functor
$$F_I : \M \lra \M, \qquad A \mapsto I\tt_{O_B}A,$$
where the action of $O_B$ on $I\tt_{O_B}A$ arises from its left action on $I$.
Then, for any $\L$ as above, we have
$j_{I\L} = F_I\circ j_\L,$
since $I\L \simeq I\tt_{O_B}\L$.

As a consequence of the previous observations, we can and do assume from now on
that $\L$ is the given order $O_B$, which we will frequently identify with the
corresponding $\OK$-lattice in $\kay^2$ via (\ref{equiiso}).

\section{\bf Special endomorphisms}\label{section.spec.endo}

In this section, we determine the endomorphism rings $\End(A)$ and
$\End(A,\iota_B)$
for $(A,\iota_B)$ coming from $(E,\iota)$ by the construction of the previous
section in various cases.
We then determine the space of special endomorphisms
$$L(A,\iota_B) = \{\ x\in \End(A,\iota_B)\mid \tr(x)=0\ \}.$$
We write $V(A,\iota_B) = L(A,\iota_B)\tt_\Z\Q$.

For $(E,\iota)/S$,  let
$O_E = \End(E/S)$ and let $\BB = \End^0(E/S)$, 
so that $O_E$ is an order in $\BB$,
with a given embedding
$i'=\iota:\OK \rightarrow O_E$.
First we consider rational endomorphisms.
As explained above,  we have homomorphisms\footnote{Here (\ref{equiiso}) depends
on the choice of $\delta$ and (\ref{quasiiso}) depends on the choice of the
free lattice $\L_0$. Both of these choices are fixed at the outset.}
\begin{equation}\label{longhoms}
\kay \overset{i}{\lra} B\overset{ \text{(\ref{equiiso})}}\lra M_2(\kay)\overset{i'}{\lra} M_2(\BB) = \End^0(E\times E)
\overset{\text{(\ref{quasiiso})}}{=} \End^0(A).
\end{equation}
and  $C:=\End^0(A,\iota_B)$, is the centralizer of $i'(B)$ in $M_2(\BB)$.
Explicitly, a simple computation shows that
\begin{equation}
C = \{\  [\a,\bbold] = \begin{pmatrix}\a&\bbold\\ \kappa \bbold&\a\end{pmatrix}
\in  M_2(\BB)\mid \a\in i'(\kay), \ \bbold \in V(E,\iota) \ \},
\end{equation}
where $V(E,\iota) = L(E,\iota)\tt_\Z\Q$ for
$$L(E,\iota) = \{ x\in \End(E/S)\mid x\,\iota(\a) = \iota(\bar\a)\,x\ \}$$
the lattice of special endomorphisms of $(E,\iota)$.
Thus
\begin{equation}\label{ratspecial}
V(A,\iota_B) = \{\ [\a,\bbold]\in C\mid \tr(\a)=0\ \},
\end{equation}
and $L(A,\iota_B)$ is the space of trace zero elements in the order
$O_C = \End(A,\iota_B)$ in $C$.
Note that, for $x = [\a,\bbold]\in V(A,\iota_B)$,
\begin{equation}\label{ssnorm}
-x^2 = N(\a)+\kappa N(\bbold).
\end{equation}
From now on, we will suppress $i'$ from the notation.

\begin{rem}  The construction just explained, based on (\ref{longhoms}), is functorial in $E/S$.
In particular, it provides us with idempotents
$$e_1= \begin{pmatrix}1&{}\\{}&0\end{pmatrix}$$
and $e_2 = 1_2-e_1$ in $\End^0(A/S)$. If a special endomorphism $x\in L(A,\iota_B)$
is given, then the corresponding `components' are $\bbold = e_1 x e_2$ and $\a= e_1 x e_1$ in
$\End^0(E/S)$, and their construction is functorial, hence, for example, commutes with any base change.
\end{rem}

\begin{lem}\label{ordinary}
 Suppose that $(E,\iota)/S$ is ordinary, i.e., that $\End(E/S)=\OK$.  Then\hfb
(i) $\End(A/S) = O_\L$. \hfb
(ii) $\End^0(A,\iota_B) = \kay$ and $\End(A,\iota_B) = \OK$. \hfb
(iii)  $L(A,\iota_B) = \{\ \a\in \OK\mid\tr(\a)=0\,\}.$
\end{lem}
\begin{proof} In this case, $\End(A/S)$ is an order
in $M_2(\kay)$ containing the maximal order $O_\L$.  This proves (i).
Since $V(E,\iota)=0$, parts (ii) and (iii) are then immediate from (i) and (\ref{ratspecial}).
\end{proof}

In general, for a CM elliptic curve $E/S$, we write $\End^0(E) = \BB$, so that $\End(E)$ is an order in $\BB$.
Notice that the matrix $\gbold$ of (\ref{lattice.def}) depends only on the lattices $\L$ and $\L_0$ and not on $E$.
We view $\End(A)$ and $\End(A_0)$
as orders in $\End^0(A)=\End^0(A_0)= M_2(\BB)$.
Then
$$\End(A) = \gbold M_2(O_{E})\gbold^{-1},$$
and, similarly,
$$O_C=\End(A,\iota_B) = C\cap \gbold M_2(O_{E})\gbold^{-1}.$$

We now describe $O_C$ and the space of special endomorphisms more explicitly.

\begin{prop}\label{special.endos}  Let $O_B = O_{\frak a, \l, B}$ for $\frak a$ and $\l$ as in
section \ref{section.maxorders}, and let $\gbold$ be given by (\ref{lattice.def}). In particular,
$\frak a = (\abold)$. \hfb
(i) The order $O_C = \End(A,\iota_B)$ in $C$ is the set of all
$$[\a,\bbold]=\begin{pmatrix}\a&\bbold\\ \kappa\bbold&\a\end{pmatrix}\in M_2(\BB)$$
such that
$$\bbold \in L(E,\iota)\, \bar{\mathfrak a}^{-1},\quad \a\in \d_2^{-1},
\quad\text{and}\quad \a+\bbold\,\l'\in O_{E}.$$
Here $\l'$ is an element of $\d_2^{-1}\bar{\frak a}$ whose image in
$\d_2^{-1}\bar{\frak a}/\bar{\frak a}$ coincides with that of $\lbold' = \bar\abold\,\abold^{-1}\,\lbold$.
\end{prop}
\begin{proof} By (\ref{lattice.def}),
$$\gbold = \abold^{-1} \begin{pmatrix}1&{}\\\lbold&\abold\end{pmatrix},$$
so that
$$
\gamma=\begin{pmatrix}\a&\bbold\\
\kappa\bbold&\a\end{pmatrix}\in C\subset M_2(\BB),$$
lies in
$O_C$ if and only if $\gbold^{-1} \gamma \gbold \in
M_2(O_E)$, i.e.,
\begin{equation} \label{eqY.1}
\begin{pmatrix}
 \a+\bar\lbold\tilde\bbold&\bar{ \abold}\tilde\bbold\\
 \nass
 \abold^{-1} (\kappa-\lbold\bar\lbold)\tilde\bbold
 & \a-\lbold\,\abold^{-1}\bar{\abold}\tilde\bbold
 \end{pmatrix} \in M_2(\widehat O_E),
 \end{equation}
 with $\tilde\bbold = \abold\bar\abold^{-1}\,\bbold$.

The conditions that the off diagonal entries lie in $\widehat O_{E}$ are then
$$
\tilde\bbold\, \abold\in \widehat O_{E}, \qquad \tilde\bbold\,
\bar{\abold}^{-1} (\kappa -\lbold \bar\lbold) \in \widehat O_E.
$$
A little case by case calculation gives
\begin{equation}
\min (\ord_w(\abold), \ord_w(\bar{\abold}^{-1} (\kappa -\lbold
\bar\lbold)))=\ord_w(\abold),
\end{equation}
for each finite place $w$ of $\kay$.
Thus,  the off diagonal condition is simply
\begin{equation}\label{eqY.2}
\tilde\bbold\in \widehat O_{E}\,\frak a^{-1}, \qquad\text{i.e.}\quad
\bbold \in O_{E}\, \mathfrak a^{-1} \abold\bar{\abold}^{-1} = O_{E}\,\bar{\mathfrak a}^{-1}.
\end{equation}

The diagonal entries  in (\ref{eqY.1}) lie in
$\widehat O_E$ if and only if
$$
\alpha + \bar\lbold \tilde\bbold \in \widehat O_E,  \qquad
\hbox{and} \qquad \alpha -\lbold  \abold^{-1}\bar{\abold} \tilde\bbold
\in  \widehat O_E.
$$
But the second condition is a consequence of the first condition and
(\ref{eqY.2}). Indeed,
\begin{align*}
\alpha -\lbold  \abold^{-1}\bar{\abold} \tilde\bbold &=\alpha +
\bar\lbold \tilde\bbold - \tr(\lbold \abold^{-1})
\tilde\bbold \abold.
\end{align*}
When $w\nmid \d_2$, so that $\lbold_w=0$, there is
nothing to check. When $w|\d_2$,
$$\ord_w (\lbold\abold^{-1}) = \ord_v(\kappa) - \ord_w(\abold) =- \ord_w(\d_2),$$
Thus, $\tr(\lbold \abold^{-1})$ is integral, while, by $(\ref{eqY.2})$,
$(\tilde\bbold\,\abold)_w \in (O_E)_w$.
\end{proof}

\section{\bf The pullback on arithmetic Chow groups}\label{pullback.I}

In this section, we determine the cycle $j_\L^*(\ZZ(t))$ on $\Cal C$
where, for $t\in \Z_{>0}$,  $\ZZ(t)$ is the special cycle on $\M$ defined in \cite{KRYcompo},
\cite{KRYbook}. This pullback is defined as the fiber product
 \begin{equation}\label{fiberproduct}
\begin{matrix}
j^*_\L(\ZZ(t)) &\lra &\ZZ(t)\\
\nass
\snass
i\downarrow&{}&\downarrow\\
\nass
\Cal C&\overset{j_\L}{\lra}&\M.
\end{matrix}
\end{equation}
Thus, $j^*_\L(\ZZ(t))$ is the stack over $\text{\rm Sch}/\OK$ which associates to a base scheme $S$
the category of collections $(E,\iota,A,\iota_B,x)$ with $(E,\iota)$ in
$\Cal C(S)$, $(A,\iota_B)$ in $\M(S)$ with $(A,\iota_B) \simeq j_\L(E,\iota)$
and $x\in V(A,\iota_B)$ a special endomorphism with $Q(x) = -x^2 = t$.

We assume that $\Q(\sqrt{-t})$ and $\kay$ are distinct, so that the images of $\Cal C$ and $\ZZ(t)$
are disjoint on the generic fiber $\M_\Q$. It follows that the generic fiber of
$j_\L^*(\ZZ(t))$ is empty.
Since $\ZZ(t)$ is relatively representable over $\M$ by an unramified morphism,
\cite{KRYbook}, (3.4.3), the same is true for $j_\L^*(\ZZ(t))$ over $\Cal C$.
In particular, the morphism $i$ is finite and unramified, and
the coarse moduli scheme corresponding
to $j_\L^*(\ZZ(t))$ is an artinian scheme.

On the other hand, recall that in Definition~\ref{ZZdef}, we introduced the moduli stacks $\ZZ_{\Cal C}(m,\frak a, \l,r)$
over $\Cal C$, where we now add the subscript $\Cal C$ to distinguish them from the
cycles $\ZZ(t)$ on $\Cal M$. The main result of this section is the following.
\begin{prop}\label{geopullback}  For $t>0$ with $\Q(\sqrt{-t})\ne\kay$,
\begin{equation}\label{onestar}
j_\L^*(\ZZ(t)) = \coprod\limits_{\substack{ \a\in \d_2^{-1} \\ \snass \tr(\a)=0\\ \snass N(\a)<t}}
\ZZ_{\Cal C}\big(\frac{\scr|\Delta|}{\scr D}(t-N(\a)), \bar{\frak a}, \l',r_\a\big),
\end{equation}
where $r_\a$ is the image of $\a$ in $\d^{-1}/\OK$ and
$\l' \in \d_2^{-1}\bar{\frak a}/\bar{\frak a}\subset \d^{-1}\bar{\frak a}/\bar{\frak a}.$
\end{prop}

Here  we view the union on the right side of (\ref{onestar}) as a disjoint union of stacks indexed by $\a$.
No terms with $N(\a)=t$ can occur, due to our assumption
that $\kay\ne \Q(\sqrt{-t})$.

\begin{proof} First we define a functor from $j^*_\L(\ZZ(t))$ to union on the right side.
Given $(A,\iota_B,E,\iota,x)$ over $S$, we have components $\a = e_1xe_1\in \d_2^{-1}$ and
$\bbold = e_1 x e_2\in L(E,\iota)\bar{\frak a}^{-1}$ with $\tr(\a)=0$ and
$\a+\bbold \l'\in O_E$.  Moreover, $t= N(\a)+\kappa N(\bbold)$. The collection
$(E,\iota,\bbold)$ is then an object of $\ZZ_{\Cal C}(m,\bar{\frak a},\l',r_\a)$
over $S$, where
$r_\a$ is the image of $\a$ in $\d^{-1}/\OK$ and
$$m = N(\frak a) N(\bbold) = \frac{|\Delta|}{D}(t-N(\a)).$$
Note that $N(\bbold)\ge 0$, so that $N(\a)\le t$.
This construction is functorial.
Conversely, over a base $S$, an object of the right side of (\ref{onestar}) in the term with index $\a$ is a collection $(E,\iota,\bbold)$
in $\ZZ_{\mathcal C}(m,\bar{\frak a},\l',r)$ where an $\a\in \d_2^{-1}$ is {\it given} by the index,
$r= r_\a$, and $m= \frac{\scr|\Delta|(t-N(\a))}{\scr D}$.
Then we obtain $(A,\iota_B,E,\iota,x)$ by taking $A= j_\L(E)$ and $x=[\a,\bbold]$.
These constructions are inverses of each other.
\end{proof}

\section{\bf Pullbacks of Green functions}

In this section, we will compute the pullback $j_\L^*\Xi(t,v)$
of the Green function associated to the cycle $\ZZ(t)$ and its contribution
to the arithmetic degree $\degh(j_\L^*\ZH(t,v))$. We continue to assume that
the cycles $j_\L(\Cal C)$ and $\ZZ(t)$ are disjoint on the generic fiber.

We begin by giving a more intrinsic description of the Green function $\Xi(t,v)$ defined in
\cite{kudla.annals} and sections 3.2 and 3.4 of \cite{KRYbook} for any $t\in \Q^\times$ and $v\in \R^\times_+$. We view $\Xi(t,v)$ as
a real valued function on $\Cal M(\C)$; its value on an object $(A,\iota_B)$ of
$\Cal M(\C)$ is defined as follows.  Let $T_e(A)$ be the tangent space at the identity and let
$T_e(A^\top)$ be the underlying real vector space, where $A^\top$ is the underlying real torus.
For $x\in \End(T_e(A^\top))$, let $|N(x)| = |\det(x)|^{\frac12}$. Write
\begin{equation}\label{AdJpm}
\End(T_e(A^\top),\iota_B) = U^++U^-,
\end{equation}
where $U^+$ (resp. $U^-$) is the space of complex linear (resp. anti-linear) endomorphisms
of $T_e(A^\top)$ commuting with the action of $O_B$. If
$$\xx\in V(A^\top,\iota_B) = \{ \ \xx\in \End(A^\top,\iota_B)\mid \tr(\xx)=0\ \}$$
is a special endomorphism of $A^\top$ and $x\in \End(T_e(A^\top),\iota_B)$ is the induced endomorphism of
$T_e(A^\top)$,  let $\pr_-(x)$ be its $U^-$-component.  Then
we let
\begin{equation}\label{newgreen}
\Xi(t,v)(A,\iota_B) = \sum_{\substack{\xx\in L(A^\top,\iota_B)\\ \snass Q(x) = t\ }} \b_1(4\pi v |N(\pr_-(x))|),
\end{equation}
where $\b_1(r)= - \Ei(-r)$ is the exponential integral, \cite{KRYbook}, (3.5.2).
Here recall that the quadratic form $Q$ on $V(A^\top,\iota_B)$ is defined by
$-x^2 = Q(x) \,\text{id}_A$.
This sum converges if there is no element  $\xx\in V(A^\top,\iota_B)$
with $Q(x)=t$ and  $\pr_-(x) =0$.  Note that the image in $\M(\C)$ of the cycle $\ZZ(t)(\C)$ consists precisely of those $(A,\iota_B)$' s
for which such a {\it holomorphic} special endomorphism $\xx$  exists,
so that the function $\Xi(t,v)$ is well defined
outside of this cycle. In particular, if $t<0$, then there are no such holomorphic endomorphisms
and $\Xi(t,v)$ is a smooth function on all of $\Cal M(\C)$.

The relation between this description and that given in \cite{KRYbook} arises as follows.
Writing $A(\C) \simeq T_e(A)/L$
where $L$ is a lattice, we can choose an isomorphism
\begin{equation}\label{ABiso}
B_\R = B\tt_\Q\R \isoarrow T_e(A), \qquad O_B\isoarrow L.
\end{equation}
This is unique up to right multiplication by an element of $O_B^\times$.  Note that
$$B_\R \isoarrow \End(T_e(A^\top),\iota_B), \qquad b \mapsto (x\mapsto x b^\iota),$$
and, under this isomorphism,
$$O_B\isoarrow \End(A^\top, \iota_B).$$
Also, $|N(x)|$ is the absolute value of the reduced norm of $x$.
The complex structure on $T_e(A)$ is then given by right multiplication by an element
$J\in B_\R^\times$, and the decomposition (\ref{AdJpm}) becomes
$B_\R = U^+ + U^-$
where $U^\pm$ is the $\pm1$-eigenspace of $\text{\rm Ad}(J)$.
If $V$ is the space of trace zero elements in $B$ with inner product $(x,y) = \tr(xy^\iota)$, then
$U^-\subset V(\R)$ is a negative $2$-plane,  oriented by the action of $J$, i.e.,
an element $z$ of the space $D$.  The $U^-$-component of $x\in O_B\cap V$ is denoted by $\pr_z(x)$
in \cite{KRYbook}, and
$|(\pr_z(x),\pr_z(x))| = 2 |N(\pr_z(x))|$.
We then have
$$\Xi(t,v)(A,\iota_B) =
\Xi(t,v)(z) = \sum_{\substack{x\in O_B\cap V \\ \snass N(x) = t\ }} \b_1(2\pi v |(\pr_z(x),\pr_z(x))|),$$
as an $O_B^\times$-invariant function on $D$, with singularities at the points $z$ for which $\pr_z(x)=0$
for some $x$ in $O_B\cap V$ with $Q(x)=t$.

We now determine the pullback of this function to $\Cal C(\C)$ under $j_\L:\Cal C(\C) \rightarrow \M(\C)$.
The analysis of special endomorphism
made in section~\ref{section.spec.endo} is based on the functorial properties of the Serre
construction. Hence it applies without change to the real tori $E^{\top}$ and $A^{\top}$
underlying  $(E,\iota)/\C$ and $A = j_\L(E)$, and yields the following result.
\begin{prop} Let $\BB = \End^0(E^\top)$, $O_{E^\top} = \End(E^\top)$ and
$$L(E^\top,\iota) = \{\ x\in \End(E^\top)\mid x \iota(\a) = \iota(\bar\a) x \ \}.$$
Then
$$
\End(A^\top,\iota_B) =
  \{\ [\a,\bbold]\in M_2(\BB)\mid \bbold \in L(E^\top,\iota)\, \bar{\mathfrak a}^{-1},\ \a\in \d_2^{-1},\
 \a+\bbold\,\l'\in O_{E^\top}\ \}.
$$
If $x=[\a,\bbold]$ is a special endomorphism, i.e., if $\tr(\a)=0$, then
$Q(x) = N(\a) +\kappa N(\bbold),$
where $-\bbold ^2=N(\bbold)\,\text{\rm id}_E$.
\end{prop}

Using this in (\ref{newgreen}), we obtain the following formula for the pullback.
\begin{prop}  Assume that $t\ne 0$ and that $\Q(\sqrt{-t}) \ne \kay$. Then
$$\Xi(t,v)(j_\L(E,\iota)) = \sum_{\substack{\a\in \d_2^{-1}\\ \snass \tr(\a)=0\\ \snass N(\a)>t}}
\ZZ_{\Cal C}\big( \frac{\scr|\Delta|}{\scr D}(t-N(\a)),D_2 v ; \bar{\frak a}, \l',r_\a\big)(E,\iota),$$
where
$\ZZ_{\Cal C}(m,v;\frak a, \l, r)$ is the function on $\Cal C(\C)$ defined in (\ref{arch.defZZ}).
\end{prop}
\begin{proof}
Note that, if $\xx = [\a,\bbold] \in V(A^\top,\iota_B)$, then $N(x_-) = \kappa \,N(\bbold)$.
Also note that $N(\bbold)\le 0$, and that $\bbold \ne0$ due to our assumption
that $\Q(\sqrt{-t})\ne \kay$. Recall that $\kappa = N(\frak a)D/|\Delta|$. Thus
\begin{align*}
\Xi(t,v)(A,\iota_B) &= \sum_{\substack{\xx=[\a,\bbold]\in \End(A^\top,\iota_B)\\ \snass Q(\xx)=t}} \b_1(4\pi v \frac{N(\frak a)D}{|\Delta|}|N(\bbold)|)\\
\nass
{}&= \sum_{\substack{\a\in \d_2^{-1}\\ \snass \tr(\a)=0\\ \snass N(\a)>t}} \quad
\sum_{\substack{\bbold \in L(E^\top,\iota)\bar{\frak a}^{-1}\\
\snass \a+\bbold \l' \in O_{E^\top}\\ \snass N(\a)+\kappa N(\bbold)=t}} \b_1(4\pi v |\frac{m D_2}{N(\d_2)}|)\\
\nass
{}&=\sum_{\substack{\a\in \d_2^{-1}\\ \snass \tr(\a)=0\\ \snass N(\a)>t}}
\ZZ_{\Cal C}\big( \frac{\scr|\Delta|}{\scr D}(t-N(\a)), D_2 v;\bar{\frak a}, \l',r_\a)(E,\iota),
\end{align*}
where, in the second line,
$N(\bbold) = \kappa^{-1}(t-N(\a)) = m\,N(\frak a)^{-1}$,
so that
$$m = \frac{|\Delta|}{D}(t-N(\a)).$$
Here recall that
$\Delta(\l') = N(\d_2)$.

\end{proof}

\section{\bf The main formula for $\degh\, j^*_\L(\phih(\tau))$.}

We now give a formula for the pullback
$$
\degh\, j^*_\L(\phih(\tau)) = \sum_{t} \degh\, j^*_\L(\Zh(t,v))\,q^t,
$$
which, by the results of \cite{KRYbook}, is a modular form of weight $\frac32$ and level $4D(B)_o$,
where $D(B)_o$ is the odd part of $D(B)$.
To express the result,  we introduce the theta functions of weight $\frac12$, defined by
$$\theta(\tau;r)= \sum_{\substack{\a\in \d^{-1}\\
\snass \tr(\a)=0\\ \snass \a\equiv r\!\! \!\mod \OK}} q^{N(\a)}\qquad \text{for\ }
r\in \d^{-1}/\OK.$$

\begin{theo}\label{theo.weak.main}  Suppose that $t\ne 0$ and that $\Q(\sqrt{-t})\ne \kay$. Then the quantity
$\degh\, j^*_\L(\Zh(t,v))$ is the coefficient of $q^t$ in the modular form
\begin{equation}\label{twostars}
\sum_{\substack{r\in \d_2^{-1}/\OK\\ \snass \tr(r)=0}} \theta(\tau;r)\,\widehat{\phi}_{\Cal C}(D_2\tau; \bar{\frak a}, \l',r),
\end{equation}
where $\widehat{\phi}_{\Cal C}(\tau;\frak a,\l,r)$ is the generating function defined in (\ref{Cgenfun})
and $D = D(B) = D_1D_2$, as in section~\ref{section.maxorders}.
\end{theo}

If we assume disjoint ramification, then we have the following stronger result,
whose proof will be completed in the next section.

\begin{theo} \label{theo.main.strong} Assume that $\Delta$ and $D(B)$ are relatively prime. Then
$$j^*_\L(\phih(\tau)) =\sum_{\substack{r\in \d^{-1}/\OK\\ \snass \tr(r)=0}}
\theta(\tau;r)\,\widehat{\phi}_{\Cal C}(D(B)\tau; \bar{\frak a}, \l',r),$$
where $\widehat{\phi}_{\Cal C}(\tau;\frak a,\l,r)$ is the generating function defined in (\ref{Cgenfun}).
\end{theo}

\begin{rem}
Note that when, in addition, $2\nmid \Delta$, then by Theorem~\ref{maintheo1}, we have
\begin{equation}
\sum_{\substack{r\in \d_2^{-1}/\OK\\ \snass \tr(r)=0}} \theta(\tau;r)\,\widehat{\phi}_{\Cal C}(D_2\tau; \bar{\frak a}, \l',r)
=-\frac12\,\frac{\d}{\d s}\bigg(\sum_{\substack{r\in \d_2^{-1}/\OK\\ \snass \tr(r)=0}}
\theta(\tau;r)\, E^{*}(D_2\tau, s;\bar{\frak a},\l',r)\bigg)\bigg\vert_{s=0}.
\end{equation}
\end{rem}

\begin{proof}[Proof of Theorem~\ref{theo.weak.main}]
By the results of the previous sections, if $t\ne 0$ and $\Q(\sqrt{-t})\ne \kay$, the quantity
$\degh\, j^*_\L(\Zh(t,v))$ is the sum of the terms
\begin{equation}\label{alphasmall}
\sum_{\substack{ \a\in \d_2^{-1} \\ \snass \tr(\a)=0\\ \snass N(\a)<t}}
\degh\,\ZZ_{\Cal C}\big(\frac{\scr|\Delta|}{\scr D}(t-N(\a)), \bar{\frak a}, \l',r_\a\big),
\end{equation}
and
\begin{equation}\label{alphabig}
\sum_{\substack{\a\in \d_2^{-1}\\ \snass \tr(\a)=0\\ \snass N(\a)>t}}
\degh\,\ZZ_{\Cal C}\big( \frac{\scr|\Delta|}{\scr D}(t-N(\a)),D_2 v; \bar{\frak a}, \l',r_\a\big).
\end{equation}
The term here for a fixed $\a$ is the coefficient of $q^{\frac{D_2 m}{\Delta(\l')}}$ in $\widehat{\phi}_{\Cal C}( D_2\tau;\bar{\frak a},\l',r)$,
where
$$ \frac{D_2 m}{\Delta(\l')} = \frac{D_2 |\Delta|}{\Delta(\l')D}(t-N(\a)) = t-N(\a),$$
since $\Delta(\l')= N(\d_2)$. Recall that $\l$ is a generator for the cyclic module $\d_2^{-1}\frak a/\frak a$.
This gives the claimed identity.
\end{proof}

For later use, we compute the remaining Fourier coefficients of the modular form (\ref{twostars}).

\begin{prop}\label{prop.last.facts} (i) The constant term of (\ref{twostars}) is
$$\degh\,\ZZ_{\Cal C}(0,D_2 v) = -  \Lambda'(1, \chi) - \frac12\,\Lambda(1, \chi)\log(D_2 v).$$
Note that $D(B)=D_2 \frac{|\Delta|}{\Dl}$, and that, when $D(B)$ and $\Delta$ are relatively prime,  $D(B) = D_2$. \hfb
(ii) Suppose that $t\in \Z_{>0}$ with $\Q(\sqrt{-t}) = \kay$, and write $4t = n^2|\Delta|$. Then the $t$-th Fourier coefficient of (\ref{twostars}) is
the sum of the following terms: \hfb
(a)  The contribution where $\a= \pm\sqrt{t}$,
$$2\ \degh\,\ZZ_{\Cal C}(0,D_2 v).$$
(b) Terms given by (\ref{alphasmall}),  where $N(\a)<t$, and (\ref{alphabig}), where $N(\a)>t$.
\end{prop}
\begin{proof} Part (ii) is immediate. To prove (i), observe that the constant
term of (\ref{twostars}) is $\degh\,\ZZ_{\Cal C}(0,D_2 v)$, the contribution of the $\a=0$ term,
together with the sum
\beq\label{extra.terms} \sum_{\substack{\a\in \d^{-1} \\ \snass \tr(\a)=0}} \degh\,\ZZ(m,D_2v; \bar{\frak a},\l',r_\a),\eeq
where
$- N(\a) = D_2 \mm.$
Note that $\End^0(E^\top) = M_2(\Q) = \kay+\kay j$, where $j\in L(E^\top, \iota)$, with $(\Delta, j^2)_p=1$
for all primes $p\le \infty$. Then, in Definition~\ref{def.ztop},  $\bbold =\b j$, with $\b\in \kay$, and
$Q(\bbold) = - N(\b) j^2 = m/N(\frak a)$. Recalling that $N(\frak a) = \kappa |\Delta|/D(B)$, we have
$$N(\a) = \frac{a^2}{|\Delta|} = \frac{D(B)}{|\Delta|}\,N(\b)\,j^2\,\kappa \frac{|\Delta|}{D(B)} = N(\b)\,j^2\,\kappa.$$
But then, for any $p\le\infty$, the Hilbert symbol has value
$$1= (\Delta, N(\a))_p = (\Delta, N(\b)\,j^2\,\kappa)_p = (\Delta,\kappa)_p.$$
Since $(\Delta,\kappa)_p=-1$ for $p\mid D(B)$,  this contradicts the fact that $B$ is a division algebra.
Thus,  no such $\bbold$ can exist
and the sum (\ref{extra.terms}) is empty.
\end{proof}

\section{\bf Arithmetic adjunction and $j^*_\L(\ZH(0,v))$}\label{section.adjunction}

In this section, we determine $\degh\,j_\L^*\ZH(t,v)$ in the case where $\kay=\Q(\sqrt{-t})$ so that $j_\L(\Cal C)$
and $\ZZ(t)$ have common components, using the arithmetic adjunction
formula, discussed in section 2.7 of \cite{KRYbook}. We also determine the
arithmetic degree $\degh\,j_\L^*\ZH(0,v)$ of the pullback of the constant term $\ZH(0,v)$
of the generating function.

First consider the constant term.
\begin{prop}
$$
\degh\, j_\L^*\ZH(0,v) = - \L'(1,\chi) - \frac12\,\L(1,\chi)\,\log(v D(B)).
$$
\end{prop}
\begin{proof}
Recall that, \cite{KRYbook}, (3.5.7),
$$\ZH(0,v) = -\widehat\o - (0,\log(v D(B))) \in \CH^1(\Cal M),$$
where $\widehat\o$ is the Hodge bundle, metrized as in \cite{KRYcompo},p. 987.  Thus,
$$\degh\, j_\L^*\ZH(0,v) = - 2 \,\frac{2h_{\smallkay}}{w_{\smallkay}}\,h^*_{Fal}(E) -\deg_\Q(\Cal C)\cdot\log(v D(B)).$$
Here the initial factor of $2$ comes in because the pullback of the Hodge line bundle $\omega$
of $\Cal M$, cf. section~3.3 of \cite{KRYbook},  is the square of the Hodge line bundle for the moduli space  $\Cal C$ of
CM elliptic curves. The factor $2h_{\smallkay}$ arises from the fact that the
Faltings height $h^*_{Fal}(E)$ is given by the arithmetic degree of the Hodge line bundle divided by the degree
of the Hilbert class field. Finally, the factor $w_{\smallkay}$ in the denominator arises from the
stack.  Also note that $\deg_\Q(\Cal C) = h_{\smallkay}/w_{\smallkay}$.
Using the fact that, \cite{KRYcompo}, (10.82),
\begin{equation}
2\,h^*_{\text{Fal}}(E) = \frac12\,\log|\Delta| + \frac{L'(1,\chi)}{L(1,\chi)} -\frac12\,\log(\pi) - \frac12\,\gamma
= \frac{\L'(1,\chi)}{\L(1,\chi)},
\end{equation}
and $\L(1,\chi) = 2 h_{\smallkay}/w_{\smallkay}$, where $\L(s,\chi)$ is given by (\ref{Lschi}), we obtain
the claimed expression.
\end{proof}

Now we turn to the case where $t>0$ with $\Q(\sqrt{-t})=\kay$,
so that $\ZZ(t)$ and $j_\L(\Cal C)$
are not disjoint on the generic fiber, and we write $4t = n^2 |\Delta|$.  Then, there is a decomposition,
\cite{KRYbook}, (7.4.7),
\begin{equation}\label{Ztdecompo}
\ZZ(t) = \sum_{\substack{c\mid n\\ \snass (c,D(B))=1}} \ZZh(t:c) + \ZZv(t),
\end{equation}
of divisors on $\M$. The idea is that, in the definition of $\ZZ(t)$, we are imposing an action of the order $\Z[\sqrt{-t}]$ of conductor $n$,
while, along the component $\ZZh(t;c)$ there is an action of the order $O_{c^2|\Delta|}$ of conductor $c$.
The divisor $\ZZv(t)$ consists of vertical components in the fibers of bad reduction for $p\mid D(B)$.

\begin{lem}\label{Ztdecompo.lem} (i) If the $\OK$-lattices $\L$ and $\L'$ in $\kay^2$ are associated to
two-sided $O_B$-ideals that are inequivalent for the right translation
action of $\A_{\smallkay,f}^\times\cap N(\widehat O_B)$, then the cycles
$j_\L(\Cal C)$ and $j_{\L'}(\Cal C)$ are disjoint on the generic fiber. The number of such inequivalent $\L$'s is $2^{o(D_2)}$,
where $o(D_2)$ is the number of divisors of $D(B)$ that are inert in $\kay$.
\hfb
(ii) As a divisor on $\Cal M$,
$$\ZZh(t:1) = 2\sum_{\L'} j_{\L'}(\Cal C),$$
where $\L'$ runs representatives for the equivalence classes of $\OK$-lattices in $\kay^2$
described in (i)
\end{lem}
\begin{proof}  It suffices to check this on the generic fiber, since the cycles in question are flat over $\Spec(\OK)$.
The generic fiber of the divisor on the right side is contained in that of the one on the left.
Note that, by a slight variant of (3.4.5) of \cite{KRYbook} with the same proof,
$$\deg_\Q \ZZh(t:c) = 2^{o(D_2)+1}\,\frac{h(c^2|\Delta|)}{w(c^2|\Delta|)}.$$
Recalling that $\deg_\Q \Cal C = h_{\smallkay}/w_{\smallkay}$, we see that the degrees of the two sided coincide.
\end{proof}

\begin{rem}
There is an error at this point in Lemma~7.4.2 of \cite{KRYbook}, where the irreducibility
of $\ZZh(t:c)$ is incorrectly claimed, whereas its actual decomposition into irreducible components can be obtained by the
construction of Remark~3.4.7.
\end{rem}

Let
$$\ZZ(t)^o = \ZZ(t) - 2 \,j_\L(\Cal C),$$
and, for convenience, write $\ZZ_o=j_\L(\Cal C)$.
Let $\ZH(t,v)^o$ and $\Zh_o(tv)$ denote the corresponding classes in
$\CH^1(\M)$, where the Green functions are defined as in \cite{KRYbook}, section 3.5.
Recall that these Green functions depend on the auxillary parameter $v\in \R_{>0}$.
Then
$$\Zh(t,v) = \Zh(t,v)^o+ 2\,\Zh_o(tv),$$
and
$$\degh\,j^*\Zh(t,v) = \degh\,j^*\Zh(t,v)^o + 2\,\degh\,j^*\Zh_o(tv).$$

The quantity $\degh\,j^*\Zh_o(tv)$ is given by the arithmetic adjunction formula, (i) of Theorem~2.7.2 in \cite{KRYbook}.
More precisely,
\begin{equation}\label{adjunction}
\degh\,j^*\Zh_o(tv) = - \degh\,j^*\widehat{\obold} + \frak d_{\ZZ_o} +
\frac12 \sum_{\substack{P, P' \in \ZZ_o(\C)\\
\snass P\ne P'}} e_P^{-1}\,e_{P'}^{-1} \sum_{\substack{\gamma\in \Gamma\\ \snass z \ne \gamma z'}}
g^0_{tv}(z,\gamma z'),
\end{equation}
where, for $z$ and $z'\in D$, the function $g^0_{tv}(z,z')$ is defined in Proposition~7.3.1 and (7.3.42) of \cite{KRYbook},
and where $\frak d_{\ZZ_o}$ is the discriminant term.  In the sum, $z$ (resp. $z'$) is a preimage of $P$ (resp. $P'$)
in $D$.
Note that, as explained in the proof of Proposition~7.5.1, p.226 of \cite{KRYbook}, $\widehat{\obold}$ is the
relative dualizing sheaf on $\Cal M$ with metric determined by $-g_{tv}^0$.
Thus, by Lemma~7.5.2,
$$\widehat{\obold} = \widehat{\o} + (0,\log(4tv D(B)),$$
and
\begin{align*}
- \degh\,j^*\widehat{\obold} &=- \degh\,j^*\widehat\o - \deg_\Q(\Cal C)\,\log(4tv D(B))\\
\nass
{}&= - 2 \,\frac{2h_{\smallkay}}{w_{\smallkay}}\,h^*_{Fal}(E) - \deg_\Q(\Cal C)\,\log(4tv D(B))\\
\nass
{}&= -\L'(1,\chi) - \frac12\,\L(1,\chi)\,\log(vD(B)) -\frac12\,\L(1,\chi)\,\log(n^2|\Delta|).
\end{align*}
Also, the discriminant term is simply
$$\frak d_{\ZZ_o}  = \frac{h_{\smallkay}}{w_{\smallkay}}\,\log|\Delta| =\frac12\,\L(1,\chi)\,\log|\Delta|.$$
\begin{lem} When $D(B)=D_2$,
$$-2\, \degh\,j^*\widehat{\obold}+2\frak d_{\ZZ_o}
=2\, \degh\,\ZZ_{\Cal C}(0,D_2 v) - \L(1,\chi)\,\log(n^2).$$
\end{lem}

Next we consider the finite part of $\degh\,j_\L^*\ZZ(t,v)^o$.
\begin{prop} Assume that $\Delta$ and $D(B)$ are relatively prime.
Then the finite part of $\degh\,j_\L^*\ZZ(t,v)^o$ is the sum of (\ref{alphasmall}) and the quantity
$2\,\deg_{\Q}(\Cal C)\,\log(n^2).$
\end{prop}
\begin{proof}
Here we have to take into account the
following rather subtle phenomenon.  First, note that, for  an object  $(E,\iota)$ in $\Cal C(S)$, the abelian scheme
$A = \L\tt_{\OK}E$ has a `distinguished' special endomorphism  $\xx = 1\tt\iota(\sqrt{-t})$, where $2\sqrt{-t} = n\sqrt{\Delta}$.
We utilize the notation of Chapter 7 of \cite{KRYbook} and refer the reader to that chapter for further details.

In the case where $p\nmid D(B)$,  the formal branches of $i_*\ZZ(t)$ through a point $x\in \M_p$ corresponding to
$(A,\iota_B)$ are described in Proposition~7.7.4 of \cite{KRYbook} as follows:
$$(i_*\ZZ(t))^{\wedge}_x = \sum_{\substack{y\in V(A,\iota_B)\\ \snass Q(y)=t}} \sum_{s=0}^{\ord_p(n)} \Cal W_s(\psi_y).$$
Here $i:\ZZ(t)\lra \M$ is the unramified morphism, and $\Cal W_s(\psi)$ is the quasi-canonical divisor of level $s$
associated to $\psi$, \cite{KRYbook}, p.240.  If $x$ lies in $j_\L(\Cal C)$, then
$ (j_\L(\Cal C))^\wedge_x = \Cal W_0(\psi_{y_0}),$
where $y_0$ is the distinguished special endomorphism.  Thus,
\begin{equation}\label{extra.comp.I}
(i_*\ZZ(t)^o)^\wedge_x = \sum_{\substack{y\in V(A,\iota_B)\\ \snass Q(y)=t\\ \snass y\ne \pm y_0}} \sum_{s=0}^{\ord_p(n)} \Cal W_s(\psi_y)
+ 2\sum_{s=1}^{\ord_p(n)} \Cal W_s(\psi_{y_0}),
\end{equation}
where the factor of $2$ in the second summand is due to the fact that $\Cal W_s(\psi_{y_0}) = \Cal W_s(\psi_{-y_0})$.
In particular, the component $\Cal W_0(\psi_{y_0})$ has been removed.
Recall that the local intersection number is given by, \cite{KRYbook}, Proposition~7.7.7,
$(\Cal W_0(\psi), \Cal W_s(\psi)) = m_0(p)$
where $m_0(p)=2$ if $p$ is ramified in $\kay$ and $m_0(p)=1$ otherwise.
Thus, the contribution of the terms in the second summand in (\ref{extra.comp.I}), summed over the points of $\Cal C(\bar\F_p)$, is
$$\frac{1}{w_{\smallkay}}\,2\,\ord_p(n)\,m_0(p)\,\sum_{x\in \Cal C(\bar\F_p)} \log|\kappa(x)| = 2 \,\frac{h_{\smallkay}}{w_{\smallkay}}
\,\ord_p(n^2)\,\log(p).$$
It follows that the $\log p$ part of $\degh\,j^*\Zh(t,v)^o$ is given by the $\log p$ part of (\ref{alphasmall}) together with the additional term
\begin{equation}\label{additional.term}
2 \deg_{\Q}(\Cal C)\,\ord_p(n^2)\,\log p.
\end{equation}

Next suppose that $p\mid D(B)$.  Again, we begin with a calculation on
$\M$ and consider both inert and ramified $p$.
Now  the formal branches of $i_*\ZZ(t)$ through a point $x\in \M_p$ corresponding to
$(A,\iota_B)$ can be described using the results about the $p$-adic uniformization of the special cycles given section 8 of \cite{krinvent}.
We have
\begin{equation}\label{branches.pmidDB}
(i_*\ZZ(t))^{\wedge}_x = \sum_{\substack{y\in V(A,\iota_B)\\ \snass Q(y)=t}} \Cal N(\psi_y)_x,
\end{equation}
where the notation, which differs slightly from that of \cite{krinvent}, is as follows.  Let $(\X,\iota_B)$ be the $p$-divisible
group of $A$ with it's $O_B$-action and let $\Cal N$ be the Rapoport-Zink space parametrizing special formal $O_B$ modules $(X,\rho)$,
cf. \cite{krinvent}, p.154, where we require the quasi-isogeny $\rho$ to have height $0$.  A special endomorphism $y$ of $(A,\iota_B)$
gives rise to a special endomorphism $\psi_y$ of $(\X,\iota_B)$, and we denote by $\Cal N(\psi_y)$ the locus in
$\Cal N$ where it deforms. This locus was denoted by $Z(j)$ in \cite{krinvent}. Finally, we denote by $\Cal N(\psi_y)_x$
the branches of $\Cal N(\psi_y)$ at the point $x$.  The decomposition (\ref{branches.pmidDB}) then follows from the
description of the $p$-adic uniformization of the special cycle given in (8.17) and (8.20) of \cite{krinvent}.

Now suppose that $x$ lies in $j_\L(\Cal C)$, and let $y_0$ be the distinguished special endomorphism.
Then
\begin{equation}\label{branches.pmidDB.II}
(i_*\ZZ(t)^o)^{\wedge}_x = \sum_{\substack{y\in V(A,\iota_B)\\ \snass Q(y)=t\\ \snass y\ne \pm y_0}} \Cal N(\psi_y)_x + 2\, \Cal N(\psi_{y_0})^\text{\rm ver}_x,
\end{equation}
where $\Cal N(\psi_{y_0})^\text{\rm ver}_x$ is the vertical part  of $\Cal N(\psi_{y_0})_x$.
Note that the divisors $\Cal N(\psi_{\pm y_0})_x^{\text{\rm hor}}$
have been omitted on the right side of (\ref{branches.pmidDB}).

First suppose that $p\ne 2$, and write $4t = n^2 |\Delta|$.  Then $\ord_p(t) = 2 \,\ord_p(n)$ (resp. $2\,\ord_p(n)+1$) if $p$ is inert (resp. ramified)
in $\kay$.  The structure of $\Cal N(\psi_{y_0})$ is then shown in the pictures on p.161 of \cite{krinvent} and described in
Proposition~4.5.  In the inert case, $\Cal N(\psi_{y_0})^\text{\rm ver}_x$
consists of a single component of the special fiber $\M_p$ with multiplicity $\ord_p(n)$, and the intersection multiplicity at $x$ of this component with
$j_\L(\Cal C)$ is $1$. In the ramified case, $\Cal N(\psi_{y_0})^\text{\rm ver}_x$ consists of $2$ components of $\M_p$ meeting at $x$,
each with multiplicity $\ord_p(n)$,
and the intersection multiplicity at $x$ of each of them with $j_\L(\Cal C)$ is $1$, cf. Lemma~4.9.
Again, it follows that the $\log p$ part of $\degh\,j^*\Zh(t,v)^o$ is given by the $\log p$ part of (\ref{alphasmall}) together with the additional term
(\ref{additional.term}).

Finally, suppose that $p=2$. In this case the structure of $\Cal N(\psi_{y_0})$ is given in the appendix to section 11 of \cite{KRYcompo}
and in section 6A.2 of \cite{KRYbook}.  Since the field $\kay = \Q(\sqrt{-t})$ spits $B$, we need only consider the
cases in which $2$ is not split in $\kay$, using the descriptions on p.187 of \cite{KRYbook}.  Again we write
$4t = n^2|\Delta|$ and $t = \varepsilon p^\a\in \Z_2$. In the inert case {\bf (2)}, there is one
`central' component of $\Cal N(\psi_{y_0})^{\text{\rm ver}}_x$. It has multiplicity $\mu= \frac{\a}2+1$ and
its intersection number with the divisor $j_\L(\Cal C)_x^{\wedge}$ is $1$.  Note that $\ord_2(\Delta)=0$ in this case, so that
$\mu=\ord_2(n)$.  In the ramified case {\bf (3)}, there are a pair of `central' components of $\Cal N(\psi_{y_0})^{\text{\rm ver}}_x$
meeting in a unique superspecial point and each having multiplicity $\mu = \frac{\a}2$. The intersection number of
$j_\L(\Cal C)_x^{\wedge}$ with each of these central components is $1$, and, since $\ord_2(\Delta)=2$,
$\mu= \ord_2(n)$.  Finally, the configuration in case {\bf (3)} is the same as in case {\bf (2)}, except that now
the multiplicity of the central components is $\mu= \frac{\a-1}2$. Since $\ord_2(\Delta)=3$, we again find that
$\mu= \ord_2(n)$.  Thus, we find the same contribution as in the other cases.
\end{proof}

\begin{proof}[Proof of Theorem~\ref{theo.main.strong}]
We simply observe that the Fourier coefficients of the pullback that are not covered by Theorem~\ref{theo.weak.main}
have now been computed and match those described in Proposition~\ref{prop.last.facts}.  Note that the archimedean
contributions include the archimedean term in the adjunction formula.
\end{proof}

\end{document}